\documentclass[12pt,reqno]{amsart}
\usepackage[numbers,sort&compress]{natbib}
\usepackage{amsfonts,bbm}
\usepackage{amssymb,color}
\usepackage{amssymb}
\usepackage{fancyhdr}
\usepackage[titletoc]{appendix}
\usepackage{enumitem}
\usepackage{amsgen}
\usepackage{amscd}
\usepackage{amsmath}
\usepackage{mathrsfs}
\usepackage{cases}
\usepackage{accents} 

\usepackage[colorlinks=true,backref]{hyperref}
\hypersetup{urlcolor=blue, citecolor=green, linkcolor=blue}

\usepackage[margin=3cm, a4paper]{geometry}

\newtheorem{thm}{Theorem}[section]
\newtheorem{cor}[thm]{Corollary}
\newtheorem{lem}[thm]{Lemma}
\newtheorem{prop}[thm]{Proposition}
\newtheorem{defn}[thm]{ \bf{Definition}}

\newtheorem{redu}[thm]{Reduction}
\theoremstyle{remark}
\newtheorem{remark}[thm]{Remark}
\newtheorem*{prob}{\textbf{Problem}}

\newcommand{\EQ}[1]{\begin{align*}\begin{split} #1 \end{split}\end{align*}}
\newcommand{\EQn}[1]{\begin{align}\begin{split} #1 \end{split}\end{align}}

\newcommand{\EQnnsub}[1]{\begin{subequations}\begin{align} #1 \end{align}\end{subequations}}

\newcommand{\enu}[1]{\begin{enumerate} #1 \end{enumerate}}

\newcommand{\Del}[1]{}

\def\norm#1{\left\|#1\right\|}

\def\normb#1{\big\|#1\big\|}

\def\normB#1{\Big\|#1\Big\|}

\def\abs#1{\left|#1\right|}

\def\absb#1{\big|#1\big|}

\def\brko#1{(#1)}
\def\brkb#1{\big(#1\big)}

\def\fbrk#1{\left\lbrace#1\right\rbrace}

\def\fbrkb#1{\big\lbrace#1\big\rbrace}
\def\fbrkbb#1{\bigg\lbrace#1\bigg\rbrace}

\def\jb#1{\langle#1\rangle}

\def\wt#1{\widetilde{#1}}
\def\wh#1{\widehat{#1}}
\def\wb#1{\overline{#1}}
\def\wtt#1{\accentset{\approx}{#1}} 

\def\pd{\partial}

\newcommand{\ra}{{\rightarrow}}
\newcommand{\hra}{{\hookrightarrow}}

\def\loe{\le}
\def\goe{\ge}
\def\lsm{\lesssim}
\def\gsm{\gtrsim}

\newcommand{\N}{{\mathbb N}}

\newcommand{\R}{{\mathbb R}}
\newcommand{\C}{{\mathbb C}}
\newcommand{\Z}{{\mathbb Z}}

\newcommand{\F}{{\mathcal{F}}}

\newcommand{\M}{{\mathcal{M}}}

\newcommand{\E}{{\mathcal{E}}}
\newcommand{\ZZ}{{\mathcal{Z}}}
\newcommand{\U}{{\mathcal{U}}}
\newcommand{\V}{{\mathcal{V}}}
\newcommand{\W}{{\mathcal{W}}}
\newcommand{\dd}{{\mathrm{d}}}

\newcommand{\TT}{{\mathcal{T}}}

\def\dx{\mathrm{\ d} x}
\def\dy{\mathrm{\ d} y}
\def\ds{\mathrm{\ d} s}

\newcommand{\re}{{\mathrm{Re}}}
\newcommand{\im}{{\mathrm{Im}}}

\def\ep{\varepsilon}
\def\al{\alpha}
\def\be{\beta}

\def\ph{\varphi}
\def\th{\theta}

\def\de{\delta}
\def\De{\Delta}

\def\ta{\tau}
\def\la{\lambda}
\def\ga{\gamma}

\newcommand{\I}{\infty}
\def\rev#1{\frac{1}{#1}}
\def\half#1{\frac{#1}{2}}

\numberwithin{equation}{section}

\begin{document}
\title[NLS]{Global well-posedness for the defocusing 3D quadratic NLS in the sharp critical space}

\subjclass[2010]{}

\author{Jia Shen}
\address{(J. Shen) School of Mathematical Sciences\\
	Nankai University\\
	Tianjin 300071, China}
\email{shenjia@nankai.edu.cn}

\author{Yifei Wu}
\address{(Y.Wu) School of Mathematical Sciences\\
Nanjing Normal  University\\
Nanjing 210046, China}
\email{yerfmath@gmail.com}
\thanks{}

\date{}

\begin{abstract}\noindent
In this paper, we prove the global well-posedness of defocusing 3D quadratic nonlinear Schr\"odinger equation
\EQ{
i\pd_t u + \frac12\De u = |u| u,
}
in its sharp critical weighted space $\F \dot H_x^{1/2}$ for radial data. 

Killip, Masaki, Murphy, and Visan \cite{KMMV17NoDEA} have proved its global well-posedness and scattering, if the $\F \dot H_x^{1/2}$-norm of the solution is bounded in the maximal lifespan. Now, we remove this a priori assumption for the global well-posedness statement in the radial case. 

Our method is based on the almost conservation of pseudo conformal energy. 
This energy scales like $\dot H_x^{-1}$, which is supercritical. We are still able to derive the global well-posedness using this monotone quantity. 
The main observation is that we can establish the local solution in supercritical weighted space when the initial time is away from the origin.
\end{abstract}

\maketitle

\tableofcontents

\section{Introduction}
We consider the Cauchy problem for the defocusing nonlinear Schr\"odinger equations (NLS):
\EQn{
\label{eq:nls}
\left\{ \aligned
&i\pd_t u + \frac12\De u = |u|^p u, \\
& u(0,x) = u_0(x),
\endaligned
\right.
}
with $p>0$ and $u(t,x):\R\times\R^d\ra\C$. 

The equation \eqref{eq:nls} has conserved mass
\EQn{\label{NLS:mass}
	M(u(t)) :=\int_{\R^d} \abs{u(t,x)}^2 \dx=M(u_0),
}
 energy
\EQn{\label{NLS:Energy}
	E(u(t)) := \frac14\int_{\R^d}\abs{\nabla u(t,x)}^2 \dx + \rev{p+2} \int_{\R^d}  \abs{u(t,x)}^{p+2} \dx=E(u_0).
}
and the pseudo conformal energy
\EQn{\label{NLS:Ps-Energy}
	P(u(t)) :=& \int_{\R^d} \abs{(x+it\nabla) u(t,x)}^2 \dd x + \frac{8t^2}{p+2} \int_{\R^d} \abs{u(t,x)}^{p+2} \dd x\\
	&\quad +\frac{dp-4}{p+2}\int_0^t\!\!s\int_{\R^d}\!\!|u(s,x)|^{p+2}\,\dd x\dd s
	=P(u_0).
}

The solution to equation (\ref{eq:nls}) is invariant under the scaling transform:
\EQ{
	u(t,x)\to u_\lambda(t,x) = \lambda^{\frac{2}{p}} u(\lambda^2 t, \lambda x) \ \ {\rm for}\ \ \lambda>0,
}
which maps the initial data as $u(0)\to u_{\lambda}(0):=\lambda^{\frac{2}{p}} u_0(\lambda x)$ for $\la>0$. Denote the scaling critical exponent
$$
s_c:=\frac d2-\frac{2}{p},
$$
then the scaling transform leaves  $\dot{H}_x^{s_{c}}$-norm invariant, that is $\|u(0)\|_{\dot H_x^{s_{c}}}=\|u_{\lambda}(0)\|_{\dot H_x^{s_{c}}}$. This Sobolev space is a nature choice to study the well-posedness since it is related to the mass and the energy conservation laws. Due to this, we usually call the equation (\ref{eq:nls}) is mass-critical when $s_c=0$ and energy-critical when $s_c=1$. 

Moreover,  the scaling transform also leaves  $\F \dot{H}_x^{-s_{c}}$-norm invariant, that is $\|u(0)\|_{\F\dot H_x^{-s_{c}}}=\|u_{\lambda}(0)\|_{\F\dot H_x^{-s_{c}}}$. Here, $\F \dot H_x^{s}(\R^d):= L_x^2(\R^d;|x|^{2s}\dd x)$ denotes the homogeneous weighted Sobolev space. In particular, when studying the mass-subcritical equation ($s_c<0$), the space $\F \dot{H}_x^{-s_{c}}$ becomes important and a natural choice,  since in this case the power of the weight is positive and thus  related to the pseudo conformal energy. One may see  \cite{KMMV17NoDEA} for more extension introduction.

The scaling critical regularity plays an important role in the study of well-posedness theory. It is well-known that the local well-posedness holds in $\dot H_x^{s}$ with $s\ge s_c$ (\cite{CW89remarks}), while there exists some ill-posedness in $\dot H_x^{s}$ when $s<s_c$ (\cite{CCT03illposed}). One of the central questions of the study of dispersive equations is that:
\begin{prob}
If the defocusing equation is globally well-posed in its critical space?
\end{prob}

The problem is related to the well-known Bourgain's conjecture in \cite{Bou00GAFA}: the defocusing NLS is globally well-posed in the inhomogeneous space $H_x^{s_c}$ when $s_c\ge 0$ and $s_c\ne 0, 1$.
As described  above, the conjecture can be naturally extended to the space $\F \dot{H}_x^{-s_{c}}$ for the mass subcritical equation.

This topic is extensively studied when the critical regularity is related to the conservation law, namely the energy critical and the mass critical cases, see \cite{Bou99JAMS,CKSTT08Annals,Dod12JAMS,Dod16AJM,Dod16Duke,KTV09JEMS,KVZ08APDE,TVZ07Duke,RV07AJM,Vis07Duke,Vis11IMRN,KV12APDE}. Now, we focus on the case when the critical regularity is not directly linked to certain conservation law, that is when $s_c\ne 0,1$. The problem becomes much harder since there is no natural a priori bound of the solution in critical space. 

A typical model of such case for NLS is the 3D cubic equation, which is $\dot H_x^{1/2}$-critical. This model is known to be globally well-posed in $H_x^s$ for some $s>1/2$, see \cite{Bou98IMRN,CKSTT02MRL,CKSTT04CPAM,Su12MRL,Dod19CJM}, and the global well-posedness in $H_x^{1/2}$ is still open. 

For the global problem in critical space, Kenig and Merle \cite{KM10TranAMS} started the study of ``conditional global well-posedness and scattering" in the case of 3D cubic NLS: if the solution is uniformly bounded in critical space within the maximal lifespan, then the equation is globally well-posed and scatters. The subject is then widely studied in the field of dispersive equations, and the related result for 3D quadratic NLS was given by Killip, Masaki, Murphy, and Visan \cite{KMMV17NoDEA}.

It is interesting to remove such a priori assumption. For mass subcritical NLS, Beceanu, Deng, Soffer, and the second author \cite{BDSW18sub-critical} gave the first global well-posedness in $\dot H_x^{s_c}$, additionally assuming that the initial data is radial and compact supported. As for the NLS with $s_c>0$ and $s_c\ne 1$, even the global well-posedness in inhomogeneous space $H_x^{s_c}$ is still open. Moreover, there are global results in the critical space based on the Lebesgue space $L_x^p$ with $p<2$ by Dodson \cite{Dodson20critical,Dod23IMRN} for the 3D cubic NLS.

In the context of nonlinear wave equations, Dodson \cite{Dod21Duke} first made an important progress in this field, giving the global well-posedness and scattering of 3D cubic equations in the sharp critical space $\dot H_x^{1/2}(\R^3)$ with radial assumption, see \cite{Dod24AJM,Dod23NLWconformal,Dod23NLWsubcritical} for further studies.

Unexpectedly, for defocusing energy supercritical NLS (that is $s_c>1$), Merle, Raphael, Rodnianski, and Szeftel \cite{MRRS22annals,MRRS22invent} constructed smooth blow-up solutions, which denied the assertion and Bourgain's conjecture in \cite{Bou00GAFA}. This makes the related problem more interesting.

Nakanishi and Ozawa \cite{NO02NoDEA} observed that there exists some symmetry between the mass supercritical NLS in $H_x^s$ and the mass subcritical ones in $\F H_x^s$. This may suggest that there are some similar difficulties between the well-known problem of the global well-posedness of 3D cubic NLS in $\dot H_x^{1/2}$ and the related problem for 3D quadratic NLS in $\F\dot H_x^{1/2}$. Moreover, the mass subcritical NLS is known to be globally well-posed in $L_x^2$ (\cite{Tsu87Funk}), whether the equation is defocusing or focusing, while the result in the homogeneous space $\F \dot H_x^{-s_c}$ remains unsolved. This motivates us to study the 3D quadratic NLS in $\F \dot H_x^{1/2}$.

The 3D quadratic NLS is also of particular interest, since the exponent of nonlinearity equals the Strauss exponent $\ga(d):=\frac{2-d+\sqrt{d^2+12d+4}}{2d}$ in 3D. 

In this paper, we give an affirmative answer to the above mentioned question, in the case of defocusing 3D quadratic NLS with radial $\F \dot H^{1/2}$-data. 
The main result is as follows.
\begin{thm}\label{thm:3dquadratic}
Let $d=3$ and $p=1$. Suppose that $u_0\in\F \dot H_x^{1/2}(\R^3)$ and $u_0$ is radial. Then, the equation \eqref{eq:nls} admits a unique solution $u$ such that $e^{-\frac12it\De}u\in C(\R;\F \dot H_x^{1/2}(\R^3))$. Moreover, the solution has additional weighted property: for any $\frac12\le s\le 1$,
\EQ{
e^{-\frac12it\De}u - u_0 \in C(\R;\F \dot H_x^{s}(\R^3)).
}
\end{thm}

\begin{remark}
Tsutsumi \cite{Tsu87Funk} has proved the global well-posedness for mass subcritical NLS in $L^2$, based on the mass conservation. However, there is no conservation law naturally in the critical space. Thus, one of our main tasks is to establish suitable monotone quantity. The method is based on the almost conservation of pseudo conformal energy, which is widely used in the previous study for scattering theory, see for example \cite{TY84BAMS,CW92CMP,GOV94poincare,SW23subNLS}, while our result seems the first time when this conservation law is applied to the proof of global well-posedness. 
\end{remark}

Moreover, the pseudo conformal energy scales like $\dot H^{-1}$, which is supercritical about the 3D quadratic NLS. One of our main observations is that we can solve the NLS in supercritical weighted space when $t$ is away from the origin. We give a precise statement in the case of 3D quadratic equation, and this principle can be surely applied to more general $d$, $p$, and $s$. 
\begin{prop}[Local well-posedness away from $t=0$] \label{thm:lwp}
Let $d=3$, $p=1$, and $0\le s\le 1$. Suppose that $e^{-\frac12it_0\De}u(t_0)\in\F \dot H_x^{s}(\R^3)$ and $t_0\ne0$. Then, there exists $I\subset \R$ such that $t_0\in I$, $0\notin I$, $|I|$ depends on $\normb{e^{-\frac12it_0\De}u(t_0)}_{\F \dot H_x^{s}(\R^3)}$, and the equation \eqref{eq:nls} admits a unique solution $u$ on $I$ such that $e^{-\frac12it\De}u(t)\in C(I;\F \dot H_x^{s}(\R^3))$.
\end{prop}

Note that $\F \dot H_x^{1/2}(\R^3)$ is the critical space, and this result covers all the subcritical, critical, and supercritical cases. Moreover, the local existence time depends on the norm of the initial data rather than the profile of the initial data, even in the critical and supercritical cases when $s\ge\frac12$. We will give the proof of this theorem in the appendix. Now, we turn back to the discussion of the main result in Theorem \ref{thm:3dquadratic}.

\begin{remark} We compare Theorem \ref{thm:3dquadratic} to the previous global results of NLS in the critical spaces. 
\begin{enumerate}
\item 
The conditional global well-posedness and scattering in \cite{KMMV17NoDEA,KMMV19DCDS} rely on the a priori assumptions in $\F \dot H_x^{-s_c}$ or $\dot H_x^{s_c}$.
\item 
Beceanu, Deng, Soffer, and the second author \cite{BDSW18sub-critical} proved the first global well-posedness for mass subcritical NLS in $\dot H_x^{s_c}$ with additional assumption that the initial data is radial and compact supported, which enables some weighted estimates for the linear solution. However, the result \cite{BDSW18sub-critical} is based on the almost conservation of mass, while we employ the modification of pseudo conformal energy, whose scaling is below the critical regularity $\dot H^{s_c}$.
\item 
The global results by Dodson \cite{Dodson20critical,Dod23IMRN} rely on the initial data in $\dot W_x^{11/7,7/6}$ and $\dot B_{1,1}^2$, in which the linear flow possesses some additional $\dot H_x^{1/2}$ supercritical estimates by the dispersive inequality. 
\item 
A common feature of the above two approaches in \cite{BDSW18sub-critical} or \cite{Dodson20critical,Dod23IMRN} is that the linear flow becomes smoother away from the origin $t=0$. In our assumption in Theorem \ref{thm:3dquadratic}, the initial data belongs in $\F \dot H_x^{1/2}(\R^3)$, where the linear flow only possesses $\dot H_x^{-1/2}$-level spacetime estimates. 
\end{enumerate}

\end{remark}

\begin{remark} We make several remarks regarding the main difficulty and the method of our result:
\begin{enumerate}
\item 
First, we employ the pseudo conformal transform to change the initial data problem in $\F \dot H_x^{1/2}$ into to a final data one in $\dot H_x^{1/2}$. Then, the key of our argument is to establish the modified energy.
\item 
The main difficult is that there is one-half-derivative gap between the initial data and the energy. However, the tools that has smoothing effect for Sch\"odinger equations, like the local smoothing or bilinear Strichartz estimates, at most reduce $1/2$-order derivative. There is no additional regularity for summation. The similar difficulty also appears in the related problem of 3D cubic NLS in $H_x^{1/2}$. 
\item Our method is also applicable to more general $d$ and $p$, but it is difficult to include all the mass subcritical region.
\end{enumerate}
\end{remark}

\subsection{Historical background}
Now, we recall the previous results related to the subject of this paper in more details.
\subsubsection{On the 3D cubic NLS}
Recall that 3D cubic NLS is $\dot H^{1/2}$-critical, and 3D quadratic NLS is $\dot H^{-1/2}$-critical. Now, we review the results about the low-regularity global well-posedness for the 3D cubic NLS, which is the most classical model in the case $0<s_c<1$.

We have learned that the equation is local well-posed in $\dot H_x^{1/2}$, while the global well-posedness and scattering hold in a smaller space $H_x^1$ \cite{Kat87poincare,LS78JFA}. Although the global well-posedness in the critical space $\dot H_x^{1/2}$ for 3D cubic NLS remains open, there are several approaches towards this subject:
\begin{enumerate}
\item 
The first one is to reduce the derivative of $H^s$. Bourgain \cite{Bou98JAM} first used the high-low decomposition method (introduced in \cite{Bou98IMRN}) to give the global well-posedness in $H_x^s$ with $s>\frac{11}{13}$. The lower bound was then improved by ``I-method" gradually in \cite{CKSTT02MRL,CKSTT04CPAM,Su12MRL}, and the best result up to now is $s>\frac{5}{7}$. Under the radial assumption, Dodson \cite{Dod19CJM} showed that the result holds for almost critical space when $s>\frac12$.
\item 
Note that $H^s$ spaces in the above mentioned results are all $\dot H_x^{1/2}$ super-critical. Dodson \cite{Dodson20critical} first gave the global well-posedness in the critical space $\dot W_x^{11/7,7/6}$, based on the observation that the linear solution becomes more regular away from $t=0$ with initial data in $L_x^p$ with $p<2$. Using this observation and the method in \cite{BDSW18sub-critical}, the authors \cite{SW20cubicNLS}  obtained the global well-posedness in $\dot H_x^{1/2}$ intersected with $\dot W^{s,1}$ for $s>\frac{12}{13}$, which is a $\dot H_x^{1/2}$-subcritical space with the order of derivative less than $1$.

The scattering was also obtained in \cite{SW20cubicNLS}, since the initial data possesses finite mass. Afterwards, Dodson \cite{Dod23IMRN} gave the global well-posedness and scattering merely in the critical space $\dot B_{1,1}^2$ with radial assumption, where the mass is infinite.
\item
Kenig and Merle \cite{KM10TranAMS} proposed the concept of ``conditional global well-posedness and scattering", namely the global well-posedness and scattering hold for the solution that is uniformly bounded in the sharp critical space $\dot H_x^{1/2}$ on the maximal existence interval. 
\item 
In the general energy subcritical and mass supercritical case, the second author collaborated with Beceanu, Deng, and Soffer  \cite{BDSW19inter-critical} constructed the global well-posedness and scattering for a class of radial large data (in the sense that the $H_x^{s}$-norm with $s<s_c$ is large) based on a new outgoing/incoming wave decomposition for Schr\"odinger equations and the deformed Fourier transform.
\end{enumerate}

\subsubsection{Global well-posedness of the mass subcritical NLS}
The local and global well-posedness for defocusing and focusing mass subcritical NLS in $L^2$ were first proved by Tsutsumi \cite{Tsu87Funk}. 

Concerning the related results in the critical space, Killip, Masaki, Murphy, and Visan \cite{KMMV17NoDEA} proved that if the critical weighted $\F \dot H^{-s_c}$-norm of the solution is uniformly bounded, then the global well-posedness and scattering hold in both focusing and defocusing cases, in the spirit of \cite{KM10TranAMS}. 

For the mass subcritical NLS in critical space $\dot H_x^{s_c}$, there exists some ill-posedness for nonradial data \cite{CCT08JFA}, and the local well-posedness and small data scattering were studied in \cite{CHO13DCDS,Hid08Funk,GW14JAM} with $p>\frac{4}{d+1}$. As for the large data case, a global well-posedness for compact radial data in $\dot H^{s_c}$ was established by Beceanu, Deng, Soffer, and the second author \cite{BDSW18sub-critical}, when $p_0(d)<p<\frac4d$ for some suitable $p_0(d)>\frac{4}{d+1}$. Moreover, Killip, Masaki, Murphy, and Visan \cite{KMMV19DCDS} proved the conditional global well-posedness and scattering in the critical Sobolev space $\dot H^{s_c}$ with radial data for defocusing case and $p>\frac{4}{d+1}$. 

\subsubsection{Scattering of the mass subcritical NLS} Next, we recall the scattering results for the mass subcritical NLS. The Strauss exponent 
\EQ{
	\ga(d):=\frac{2-d+\sqrt{d^2+12d+4}}{2d}
}
is first proposed by Strauss in \cite{Str81JFA}, and he proved the scattering for mass subcritical NLS with some smallness conditions in $L^{\frac{p+2}{p+1}}$ when $\ga(d)< p < \frac 4d$. Tsutsumi and Yajima \cite{TY84BAMS} proved the existence of scattering states for $\frac 2d< p <\frac 4d$ with initial data $u_0 \in \Sigma^1$ in the defocusing case, where the scattering holds in $L_x^2$ sense. This result was extended to $\F H^1$-data by Hayashi and Ozawa \cite{HO88poinc,HO89JFA}, see Remark 2 in \cite{HO88poinc}. In the very recently, the result was improved in the sense of $H_x^1$-scattering by Burq, Georgiev, Tzvetkov, and Visciglia \cite{BGTV21NLS}.

Additionally, the construction of inverse wave operator requires the uniqueness of scattering state. The wave and inverse wave operators were constructed on $\Sigma^1$ by Tsutsumi \cite{Tsu85poincare} for $\ga(d)<p<\frac4d$, and by Cazenave-Weissler \cite{CW92CMP} and Nakanishi-Ozawa \cite{NO02NoDEA} for the endpoint exponent $p=\ga(d)$. Furthermore, Cazenave-Weissler \cite{CW92CMP} also proved the inverse wave operator with small $\Sigma^1$-data when $p>\frac{4}{d+2}$, and the $\Sigma^s$-data case was studied by Ginibre-Ozawa-Velo \cite{GOV94poincare} and Nakanishi-Ozawa \cite{NO02NoDEA} for $p>\frac{4}{d+2s}$ with $0<s<\min\fbrk{\frac d2,p}$. In the case of $0 < s < 2$, the wave operator has been constructed by Ginibre-Ozawa-Velo \cite{GOV94poincare} on $\Sigma^s$, subjected to the condition $p>\max\fbrk{\frac 2d,\frac{4}{d+2s},s}$. Recently, the authors \cite{SW23subNLS} first obtained the large data scattering in $\F H^s$ with $s<1$ for defocusing mass subcritical NLS.

The large data scattering in $\F \dot H^s$ with $s<1$ was also studied for focusing mass subcritical NLS. In the focusing case, it is well-known that the NLS has a class of standing wave equations, which do not scatter. However, Masaki \cite{Mas15CPAA,Mas17CPDE} discovered a critical non-scattering solution in $\F H^1$ and $\F \dot H^{-s_c}$ for the mass subcritical NLS strictly below the ground state, unlike the mass critical and supercritical cases. In particular, he proved that the large data scattering holds with $\F \dot H^{-s_c}$-data below that critical solution in \cite{Mas15CPAA}.

\subsection{Ideas of proof}
Next, we summarize the main ideas and the innovations in the proof of Theorem \ref{thm:3dquadratic}.

$\bullet$ \textit{The use of pseudo conformal transform.} After applying the pseudo conformal transform, the initial data problem of 3D quadratic NLS \eqref{eq:nls} in $\F \dot H_x^{1/2}(\R^3)$ is changed into a final data problem of a nonautonomous NLS:
\EQn{\label{eq:nls-pc-intro}
i\pd_t \U + \frac12\De \U = t^{-\frac12} |\U|\U
}
in $\dot H_x^{1/2}(\R^3)$.

After the high-low frequency decomposition of the final data $\V_+=P_{\gsm N_0}\U_+$ and $\W_+=P_{\lsm N_0}\U_+$, we split the solution $\U=\V+\W$ such that $\V=e^{\frac12it\De}\V_+$, which has $\dot H^{1/2}$-level spacetime estimates, and $\W$ solves
\EQ{
i\pd_t \W + \frac12\De \W = t^{-\frac12} |\V+\W|(\V+\W).
}
We expect $\W$ belongs to $\dot H_x^{1/2}\cap\dot H_x^1$-level.

$\bullet$ \textit{The local theory.} Now, we study the local well-posedness for the final data problem of \eqref{eq:nls-pc-intro} in three stages: 
\begin{enumerate}
\item 
First, we prove the local well-posedness at $\dot H_x^{1/2}$-level. The general result for mass subcritical NLS when $\frac2d < p <\frac4d$ has been proved in \cite{KMMV17NoDEA}, thus this local well-posedness part of our result is not new. However, in order for further study, we need to obtain the boundedness of $\dot H_x^{1/2}$-level norm based on the atom space by Koch-Tataru \cite{KT05CPAM}. Moreover, Nakanishi \cite{Nak01siam} first observed that the study of mass subcritical NLS in critical space requires the Lorentz modification of Strichartz estimate. 

Therefore, one of the new points in our method is to combine the Lorentz modification of the Strichartz norms and the atom space method. For the key new estimates in the local argument, please see \eqref{eq:lorentz-strichartz-up} and \eqref{eq:lorentz-duality} below. 
\item 
Second, we update the nonlinear remainder of the local solution obtained in the former step to $\dot H_x^{1}$-level. This is based on the observation that for the second iteration
\EQ{
\int_{\I}^t e^{\frac12i(t-s)\De}(|e^{\frac12is\De}f|e^{\frac12is\De}f) \dd s \text{ in }\dot H_x^1(\R^3)\text{, with }f\in\dot H_x^{1/2}(\R^3),
}
the additional regularity acquired by the Duhamel formula of final data problem is different to the initial data one. 

If considering the initial data problem, then the second iteration term
\EQ{
\int_{0}^t e^{\frac12i(t-s)\De}(|e^{\frac12is\De}f|e^{\frac12is\De}f) \dd s\text{, with }f\in\dot H_x^{1/2}(\R^3),
}
should be at $\dot H_x^{1/2}(\R^3)$-level by scaling analysis. Moreover, by the H\"older's inequality in $t$, we can only expect the second iteration has local estimates with scaling \textbf{equal to or less than} $\dot H_x^{1/2}$-level. 

On the contrary, in the case of final data problem, the scaling of the second iteration should be \textbf{equal to or higher than} the critical regularity, which includes the desired $\dot H_x^1$-estimates.
\item 
Third, in order to extend the solution away from the infinity $t=+\I$, we need to prove a local result with initial data $\W(t_1)$ in $\dot H_x^{1/2}\cap\dot H_x^{1}$, with the existence time depends on the $\norm{\W(t_1)}_{\dot H_x^{1/2}\cap\dot H_x^1}$, rather than its profile. This is reasonable since the equations becomes subcritical in $\dot H_x^1$. The technical difficulty of the proof is similar to the above step, up to some technical treatment for the time interval. Therefore, the global well-posedness is reduced to some a priori estimate at $\dot H_x^{1/2}\cap\dot H_x^{1}$-level. 
\end{enumerate}

$\bullet$ \textit{Main idea of the nonlinear estimate in local $\dot H^1$-theory.}  The main difficulty in the local theory is to reduce the first order derivative acting on the linear solution: we need to deal with the term
\EQ{
\norm{(|\V|+|\W|)\nabla\V}_{N},
}
where $N$ denotes some dual Strichartz norm. The main obstruction is that the nonlinearity is not algebraic, and there is no additional regularity for summation if we make the frequency decomposition. The idea is combining the local smoothing effect estimate and radial Sobolev's inequality. Roughly speaking, if we neglect some technical treatment on time interval, the nonlinear term can be bounded by
\EQ{
\norm{(|\V|+|\W|)\nabla\V}_{L_{t}^{\frac43}L_x^{\frac32}} \lsm & \normb{|x|^{\frac12}(|\V|+|\W|)}_{L_t^\I L_x^6} \normb{|x|^{-\frac12}\nabla\V}_{L_{t,x}^2} \\
\lsm & (\norm{\V}_{\dot H_x^{\frac12}} + \norm{\W}_{\dot H_x^{\frac12}})\norm{\V}_{\dot H_x^{\frac12}}.
}
In fact, we should make some $\ep$-perturbation for summation of spatial localization.

$\bullet$ \textit{A priori estimate.} Next, we establish the almost conservation of the pseudo conformal energy of the nonlinear remainder $\W$: 
\EQn{\label{eq:pseudo-conformal-energy-intro}
\E(t) := \frac{1}{4}t^{\frac{1}{2}}\int_{\R^3}|\nabla\W(t,x)|^2 \dx + \frac{1}{3} \int_{\R^3}|\W(t,x)|^{3} \dx \lsm N_0.
}
by the bootstrap argument. 

First, the energy is finite near $t=+\I$. By the local result, the solution $\W$ belongs to $\dot H_x^{1/2} \cap \dot H_x^1$. Particularly, note that the $L^3$ part is exactly bounded by the critical space using Sobolev's inequality. This is what makes the 3D quadratic case special.

In order to close the energy estimate, the main term of the energy increment is
\EQ{
\int_I \int_{\R^3} |\nabla \W\cdot \nabla\V \W| \dd x\dd t,
}
where $\V$ denotes the high frequency part of the linear flow and $I$ denotes some suitable long time interval. 

Since we only consider the global well-posedness rather than scattering, it suffices to extend the solution close to the origin, therefore the temporal singularity problem does not appear as in our previous study in \cite{SW23subNLS}. Then, we can overlook the factor $t^{-\al}$ for some $\al>0$ and the length of time integral. 

Therefore, roughly speaking, we can estimate the main term by local smoothing and radial Sobolev inequality,
\EQnnsub{
\int_I \int_{\R^3} |\nabla \W\cdot \nabla\V \W| \dd x\dd t \lsm_{|I|} & \norm{\nabla \W}_{L_t^\I L_x^2} \normb{|x|^{-\frac12} \nabla\V}_{L_{t,x}^2} \normb{|x|^{\frac12}\W}_{L_{t}^\I L_x^\I} \nonumber\\
\lsm_{|I|} & \norm{\V(0)}_{\dot H_x^{\frac12}} \norm{\nabla\W}_{L_t^\I L_x^2}^2 \nonumber\\
\lsm_{|I|} & \norm{\V(0)}_{\dot H_x^{\frac12}} \E(t). \label{observation-1}\tag{Observation I}
}
However, to achieve the above heuristic idea, there is also a problem of summation of spatial discretization. 

The main difficulty originates from the local smoothing estimate holds only on one piece of dyadic decomposition $\sup_{j\in\Z}\normb{|x|^{-\frac12} \nabla\V}_{L_{t,x\sim 2^j}^2}$, or the global in space version requires a additional small factor $\normb{\jb{x}^{-\frac12-\ep} \nabla\V}_{L_{t,x}^2}$. Moreover, to make some $\ep$-perturbation for the idea in \eqref{observation-1} is not as easy as the local case, since we need to control the energy increase properly.

After the localization in space
\EQ{
\sum_{j\in\Z} \int_I \int_{\R^3} \chi_j|\nabla \W\cdot \nabla\V \W| \dd x\dd t,
}
where $\chi_j$ is a cut-off function on $\fbrk{|x|\sim 2^j}$, we adopt the following strategy to cover the logarithmic loss over the summation of $j$:
\begin{enumerate}
\item ($j\ge 1$ case) By the radial Sobolev inequality and Gagliardo-Nirenberg's inequality,
\EQ{
	\normb{|x|^{\frac12+}\W}_{L_{t}^\I L_x^\I} \lsm \norm{\W}_{L_t^\I L_x^3}^{0+} \norm{\W}_{L_t^\I \dot H_x^1}^{1-}. 
}
Since the $N_0$-increase of $L^3$-estimate is lower than $\dot H^1$-estimate, we can derive that the energy increase of $\normb{|x|^{\frac12+}\W}_{L_{t}^\I L_x^\I}$ is lower than $N_0^{1/2}$.

Therefore, we can make the following perturbation:
\EQnnsub{\label{observation-2}\tag{Observation II}
\int_I \int_{\R^3} \chi_j |\nabla \W\cdot \nabla\V \W| \dd x\dd t \lsm_{|I|} & 2^{-\ep j} \norm{\nabla \W}_{L_t^\I L_{x\sim2^j}^2} \normb{2^{-\frac12j} \nabla\V}_{L_{t,x\sim2^j}^2} \normb{2^{\frac12j+}\W}_{L_{t}^\I L_{x\sim2^j}^\I}.
}
Here, the factor $2^{-\ep j}$ serves to the summation over the range $j\ge 1$, and the increase of $N_0$ is reasonable.
\item (Various long time spacetime estimates based on atom space $U_\De^2$) Under the bootstrap hypothesis \eqref{eq:pseudo-conformal-energy-intro}, we have that $\norm{\W}_{L_t^\I \dot H_x^1} \lsm N_0^{1/2}$ and $\norm{\W}_{L_t^\I L_x^3} \lsm N_0^{1/3}$. Moreover, we denote $X^s$ as some $\dot H_x^s$-level spacetime norm based on $U_\De^2$. (see \eqref{defn:xs} for precise definition) Then, roughly speaking, we have
\EQ{
\norm{\W}_{X^{\frac12}} & \lsm  \normb{|\nabla|^{\frac12}(|\W|\W)}_{L_t^{\frac43}L_x^{\frac32}}+\text{``lower increase term"} \\
& \lsm_{|I|}  \normb{|\nabla|^{\frac12}\W}_{L_t^\I L_x^3} \norm{\W}_{L_t^\I L_x^3} +\text{``lower increase term"} \\
& \lsm_{|I|}  N_0^{\frac56},
}
and
\EQ{
\norm{\W}_{X^{1}} & \lsm  \norm{\nabla(|\W|\W)}_{L_t^{\frac43}L_x^{\frac32}}+\text{``lower increase term"} \\
& \lsm_{|I|}  \normb{|\nabla|^{\frac12}\W}_{L_t^\I L_x^2} \norm{\W}_{L_t^\I L_x^6} +\text{``lower increase term"} \\
& \lsm_{|I|}  N_0.
}
These global estimates allow us to control the spacetime norm related to local smoothing effect:
\EQ{
\normb{|x|^{-\frac12}\nabla\W}_{L_{t,x}^2} \lsm_{|I|} \norm{\W}_{X^{\frac12}} \lsm_{|I|} N_0^{\frac56};\quad \normb{|x|^{-\frac12}|\nabla|^{\frac32}\W}_{L_{t,x}^2} \lsm_{|I|} \norm{\W}_{X^{1}} \lsm_{|I|} N_0.
}

\item ($j\le 1$ case) For the $j\le 1$ case, we use the following observation
\EQnnsub{\label{observation-3}
& \int_I \int_{\R^3} \chi_j \absb{|\nabla|^{\frac32} \W|\nabla|^{\frac12}\V \W} \dd x\dd t \nonumber\\
\lsm_{|I|} & 2^{\frac12j} \normb{2^{-\frac12j}|\nabla|^{\frac32}\W}_{L_{t,x\sim2^j}^2} \normb{2^{-\frac12j} |\nabla|^{\frac12}\V}_{L_{t,x\sim2^j}^2} \normb{2^{\frac12j}\W}_{L_{t}^\I L_{x\sim2^j}^\I}. \tag{Observation III}
}
Here, $2^{\frac12j}$ is used to sum over the range $j\le 1$. We can check that the increase of each $\W$-term: $\normb{|x|^{-\frac12}|\nabla|^{\frac32}\W}_{L_{t,x}^2} \lsm N_0$ and $\normb{|x|^{\frac12}\W}_{L_{t}^\I L_x^\I}\lsm N_0^{1/2}$. This gives an increase of $N_0^{\frac32}$. Recall that $\V$ is high-frequency cut-off, that is $|\xi|\gsm N_0$, then 
\EQ{
\normb{2^{-\frac12j} |\nabla|^{\frac12}\V}_{L_{t,x\sim2^j}^2} \lsm \norm{\V_+}_{L^2} \lsm N_0^{-\frac12},
}
which can cover the additional increase of $N_0$. Moreover, to achieve the idea in  \eqref{observation-3}, we also need some technical treatment for the commutator of $\chi_j$ and $|\nabla|^{1/2}$.
\end{enumerate}

\subsection{Organization of the paper}
In Section \ref{sec:prel}, we give some notation and useful lemmas. In Section \ref{sec:proof}, we give the proof of Theorem \ref{thm:3dquadratic}: in Section \ref{sec:highlow}, we give the basic settings; in Section \ref{sec:localI}-\ref{sec:localIII}, we give the above mentioned three local results; in Section \ref{sec:energy}, we prove the a priori estimate; in Section \ref{sec:gwp}-\ref{sec:maintheorem}, we finish the proof of Theorem \ref{thm:3dquadratic}.

\vspace{2cm}

\section{Preliminaries}\label{sec:prel}

\vspace{0.5cm}

\subsection{Basic notation}
For any $a\in\R$, $a\pm:=a\pm\epsilon$ for arbitrary small $\epsilon>0$. 
For any $z\in\C$, we define $\re z$ and $\im z$ as the real and imaginary part of $z$, respectively. 
$C>0$ represents some constant that may vary from line to line. We write $C(a)>0$ for some constant depending on coefficient $a$. If $f\loe C g$, we write $f\lsm g$. If $f\loe C g$ and $g\loe C f$, we write $f\sim g$. Suppose further that $C=C(a)$ depends on $a$, then we write $f\lsm_a g$ and $f\sim_a g$, respectively. 


We use $\wh f$ or $\F f$ to denote the Fourier transform of $f$:
\EQ{
	\wh f(\xi)=\F f(\xi):= \rev{(2\pi)^{d/2}}\int_{\R^d} e^{-ix\cdot\xi}f(x)\rm dx.
}
We also define the inverse Fourier transform:
\EQ{
	\F^{-1} g(x):= \rev{(2\pi)^{d/2}}\int_{\R^d} e^{ix\cdot\xi}g(\xi)\rm d\xi.
}
Using the Fourier transform, we can define the fractional derivative $\abs{\nabla} := \F^{-1}|\xi|\F $ and $\abs{\nabla}^s:=\F^{-1}|\xi|^s\F $.  

Take a cut-off function $\phi\in C_{0}^{\infty}(\R)$ such that $\phi(r)=1$ if $0\loe r\loe1$ and $\phi(r)=0$ if $r>2$. Next, we introduce the spatial cut-off function. Denote $\chi_j(r):=\phi(2^{-j}r)-\phi(2^{-j+1}r)$ for $j\in\Z$, and $\chi_{\le j}:= \sum_{k=-\I}^j \chi_k$. Particularly, $\chi_{\le 0}(r)=\phi(r)$. We also define the fattening version $\wt  \chi_j:=\phi(2^{-j-1}r)-\phi(2^{-j+2}r)$ and $\wtt  \chi_j:=\phi(2^{-j-2}r)-\phi(2^{-j+3}r)$ with the property $\chi_j=\chi_j\wt \chi_j=\chi_j\wt \chi_j \wtt \chi_j$.

We also need the usual homogeneous dyadic Littlewood-Paley decomposition. For dyadic $N\in 2^\Z$, let $\phi_{N}(r) := \phi(N^{-1}r)$. Then, we define
\EQ{
	\ph_1(r):=\phi(r)\text{, and }\ph_N(r):=\phi_N(r)-\phi_{N/2}(r)\text{, for any }N\goe 2.
}
We define the inhomogeneous Littlewood-Paley dyadic operator: for any $N\in2^\Z$,
\EQ{
	f_{N}= P_N f := \mathcal{F}^{-1}\brko{ \ph_N(|\xi|) \hat{f}(\xi)}.
}
Then, by definition, we have $f=\sum_{N\in2^\Z}f_N$.

Given $1\loe p \loe \I$, $L^p(\R^d)$ denotes the usual Lebesgue space. For $1\le p<\I$ and $1\le q\le\I$, the Lorentz space $L^{p,q}$ is defined by its quasi-norm
\EQ{
\norm{f}_{L^{p,q}(\R^d)} := p^{\frac1q} \normb{ \la \absb{\fbrk{x\in\R^d:|f(x)|>\la}}^{\frac 1p} }_{L^q((0,\I),\frac{\dd\la}{\la})}.
}
We also define $\jb{\cdot,\cdot}$ as real $L^2$ inner product:
\EQ{
	\jb{f,g} := \re\int f(x)\wb{g}(x)\dx.
}

For any $1\loe p <\I$, define $l_N^p:=l_{N\in 2^\Z}^p$ by its norm
\EQ{
	\norm{c_N}_{l_{N\in 2^\Z}^p}^p:=\sum_{N\in 2^\Z}|c_N|^p.
}
The space $l_N^\I=l_{N\in2^\Z}^\I$ is defined by $\norm{c_N}_{l_{N\in 2^\Z}^\I}:=\sup_{N\in 2^\Z}|c_N|$. In this paper, we use the following abbreviations
\EQ{
	\sum_{N:N\loe N_1}:=\sum_{N\in2^\Z:N\loe N_1}\text{, and }\sum_{N\loe N_1}:=\sum_{N,N_1\in2^\Z:N\loe N_1}.
}

For any $0\loe\gamma\loe1$, we call that the exponent pair $(q,r)\in\R^2$ is $\dot H^\ga$-$admissible$, if $\frac{2}{q}+\frac{d}{r}=\half d-\ga$, $2\loe q\loe\I$, $2\loe r\loe\I$, and $(q,r,d)\ne(2,\I,2)$. If $\ga=0$, we say that $(q,r)$ is $L^2$-$admissible$. 
\subsection{Atom space and bounded variation space}
We recall the definitions of $U^p$ and $V^p$, and some properties used in this paper. The $U^p$-$V^p$ method was first introduced by Koch-Tataru \cite{KT05CPAM}, and we also refer the readers to \cite{HHK10Poinc,KTV14book,KT18Duke,CH18AnnPDE} for their complete theories. 

\begin{defn}
	Let $\mathcal{Z}$ be the set of finite partitions $-\I<t_0<t_1<...<t_K=\I$. 
	\enu{
		\item
		For $\fbrk{t_k}_{k=0}^K\in\ZZ$ and $\fbrk{\phi_k}_{k=0}^{K-1}\subset L_x^2(\R^d)$ with $\sum_{k=0}^{K-1} \norm{\phi_k}_{L_x^2(\R^d)}^p =1$, we call the function $a:\R\ra L_x^2(\R^d)$ given by $a=\sum_{k=1}^K \mathbbm{1}_{[t_{k-1},t_k)}\phi_{k-1}$ a $U^p$-$atom$. Furthermore, we define the atomic space
		\EQn{
			U^p(\R;L^2(\R^d)):=\fbrkb{u=\sum_{j=1}^\I \la_ja_j: a_j\ U^p\text{-atom, }\la_j\in\C\text{ with }\sum_{j=1}^\I\abs{\la_j}<\I},
		}
		with norm
		\EQn{\label{eq:upnorm}
			\norm{u}_{U^p(\R;L^2(\R^d))}:=\inf\fbrkb{\sum_{j=1}^\I\abs{\la_j}:u=\sum_{j=1}^\I \la_ja_j,a_j\  U^p\text{-atom, }\la_j\in\C}.
		}
		\item
		We define $V^p(\R;L^2(\R^d))$ as the normed space of all functions $v:\R\ra L^2(\R^d)$ such that 
		\EQn{\label{eq:vpnorm}
			\norm{v}_{V^p(\R;L_x^2(\R^d))}:=\sup_{\fbrk{t_k}_{k=0}^K\in\ZZ}\brkb{\sum_{k=1}^K\norm{v(t_k)-v(t_{k-1})}_{L_x^2(\R^d)}^p}^{1/p}
		}
		is finite, where we use the convention $v(t_K)=v(\I)=0$.  $V_{rc}^p$ denotes the closed subspace of all right-continuous $V^p$ functions with $\lim_{t\ra -\I}v(t)=0$.
		\item
		We define $U_\De^2(\R; L_x^2(\R^d))$ as the adapted normed space:
		\EQ{
			U_\De^2(\R; L_x^2):=\fbrkb{u:\norm{u}_{U_\De^2(\R; L_x^2)}:=\norm{e^{-it\De}u}_{U^2(\R;L_x^2(\R^d))}<\I}.
		}
		Similarly, $V_\De^2(\R; L_x^2(\R^d))$ denotes the adapted normed space
		\EQ{
			V_\De^2(\R; L_x^2(\R^d)):=\fbrkb{u:\norm{u}_{V_\De^2(\R; L_x^2(\R^d))}:=\norm{e^{-it\De}u}_{V^2(\R;L_x^2(\R^d))}<\I, e^{-it\De}u\in V_{rc}^2}.
		}
	}
\end{defn}

In this paper, we will use restriction spaces to some interval $I\subset \R$: $U^p(I;L_x^2(\R^d))$, $V^p(I;L_x^2(\R^d))$, $U_\De^p(I;L_x^2(\R^d))$, and $V_\De^p(I;L_x^2(\R^d))$. See Remark 2.23 in \cite{HHK10Poinc} for more details.

Note that for $1\loe p<q<\I$, the embeddings
\EQ{
	U^p(\R;L_x^2(\R^d)) \hra L_t^\I(\R;L_x^2(\R^d))\text{, }V^p(\R;L_x^2(\R^d)) \hra L_t^\I(\R;L_x^2(\R^d)),
}
and $U^p\hra V_{rc}^p \hra U^q$ are continuous.

We need the following classical linear estimate and duality formula:
\begin{lem}[\cite{HHK10Poinc}]\label{lem:upvpduality}
	Let $I$ be an interval such that $0=\inf I$. Then, for any $f\in L_x^2(\R^d)$,
	\EQn{\label{eq:upvpduality-1}
		\norm{e^{it\De}f}_{U_\De^2(I;L_x^2(\R^d))} \lsm \norm{f}_{L_x^2(\R^d)},
	} 
	and for $F(t,x)\in L_t^1L_x^2(I\times\R^d)$,
	\EQn{\label{eq:upvpduality-2}
		\normb{\int_0^t e^{i(t-s)\De} F(s) \ds}_{U_\De^2(I;L_x^2)}
		= \sup_{\norm{g}_{V_\De^2(I;L_x^2(\R^d))}=1} \abs{\int_{I}\int_{\R^d} F(t)\wb{g(t)}\dx\dd t}.
	}
\end{lem}
\subsection{Useful lemmas}
In this subsection, we gather some useful results.
\begin{lem}[Maximal Littlewood-Paley estimates, \cite{SSW23CMP}]\label{lem:linfty-littewoodpaley}
	Let $1<p<\I$ and $f\in L_x^p(\R^d)$. Then, we have
	\EQ{
		\normb{\sup_{N\in 2^\N}\absb{P_{N} f}}_{L_x^p(\R^d)} + \normb{\sup_{N\in 2^\N}\absb{P_{\loe N} f}}_{L_x^p(\R^d)} \lsm \norm{f}_{L_x^p(\R^d)}.
	}
\end{lem}
\begin{lem}[$L^pl^2$-boundedness for maximal function, \cite{Stein93book}]\label{lem:HL-boundedness}
	Let $1<p<\I$ and $\fbrk{f_N}_{N\in 2^\Z}$ be a sequence of functions such that $\norm{f_N}_{l_{N\in2^\Z}^2}\in L_x^p(\R^d)$. Then, we have
	\EQ{
		\norm{\M(f_N)}_{L_x^p l_{N}^2(\R^d\times2^\Z)} \lsm \norm{f_N}_{L_x^p l_{N}^2(\R^d\times2^\Z)}.
	}
\end{lem}

\begin{lem}[Square function's estimates, \cite{Stein93book}]\label{lem:square-function}
	Let $1<p<\I$ and $f\in L_x^p(\R^d)$. Then, we have
	\EQ{
		\norm{f_N}_{L_x^p l_{N}^2(\R^d\times2^\Z)}\sim_p \norm{f}_{L_x^p(\R^d)}.
	}
\end{lem}

\begin{remark} \label{rem:square-function-spacetime}
	Combining Minkowski's inequality and Lemma \ref{lem:square-function}, we  have that for any spacetime function $F(t,x)$, $1\le q,r,p\le 2$,
	\EQ{
		\norm{P_NF}_{l_N^2L_t^{q} L_x^r(2^\Z\times\R\times\R^d)} \lsm \norm{F}_{L_t^{q,p} L_x^r(\R\times\R^d)},
	}
	and
	\EQ{
		\norm{P_NF}_{l_N^2L_t^{q,p} L_x^r(2^\Z\times\R\times\R^d)} \lsm \norm{F}_{L_t^{q,p} L_x^r(\R\times\R^d)}.
	}
\end{remark}
\begin{lem}[Schur's test]\label{lem:schurtest}
For any $a>0$, let sequences $\fbrk{a_N}$,  $\fbrk{b_N}\in l_{N\in2^\Z}^2$, then we have
\EQ{
\sum_{N_1\loe N} \brkb{\frac{N_1}{N}}^a a_N b_{N_1} \lsm \norm{a_N}_{l_N^2} \norm{b_N}_{l_N^2}.
}
\end{lem}
\begin{lem}[Hardy's inequality]\label{lem:hardy}
	For $1<p<d$, we have that
	\EQ{
		\normb{|x|^{-1}u}_{L_x^p(\R^d)}\lsm \norm{\nabla u}_{L_x^p(\R^d)}.
	}
\end{lem}
\begin{lem}[Gagliardo-Nirenberg's inequality]\label{lem:GN}
	Let $d\goe 1$, $0<\th<1$, $0<s_1<s_2$ and $1<p_1,p_2,p_3\loe\I$. Suppose that $\rev{p_1}= \frac\th{p_2}+\frac{1-\th}{p_3}$ and $s_1=\th s_2$. Then, we have
	\EQn{\label{eq:GN}
		\norm{\abs{\nabla}^{s_1}u}_{L_x^{p_1}(\R^d)} \lsm \norm{\abs{\nabla}^{s_2}u}_{L_x^{p_2}(\R^d)}^\th \norm{u}_{L_x^{p_3}(\R^d)}^{1-\th}.
	}
\end{lem}
\begin{lem}[Radial Sobolev's inequality, \cite{TVZ07Duke}]\label{lem:radial-sobolev}
	Let $1\loe p,q\loe+\I$, $d\goe 1$, $0<s<d$, and $\al,\be\in\R$ satisfy
	\EQ{
		\al>-\frac{d}{p'}\text{, }\be>-\frac{d}{q'}\text{, }1\loe\rev p+\rev q\loe 1+s,
	}
	and the scaling condition
	\EQ{
		\al+\be-d+s=-\frac{d}{p'}-\frac{d}{q'}
	}
	with at most one of the equalities
	\EQ{
		p=1\text{, }p=\I\text{, }q=1\text{, }q=\I\text{, and }\rev p+\rev q=1+s
	}
	holding. Then, for radial $u:\R^d\ra\C$, we have
	\EQ{
		\normb{|x|^\be u}_{L_x^{q'}(\R^d)}\lsm_{\al,\be,p,q,s}\normb{|x|^{-\al}\abs{\nabla}^s u}_{L_x^{p}(\R^d)}.
	}
\end{lem}

\begin{lem}[Fractional chain rule, \cite{CW01JFA}]\label{lem:Frac_chain}
Let $G\in C^1(\C)$ such that $G(0)=0$, and there exists $\mu\in L^1([0,1])$ and $H\in C(\C)$ such that for any $0\le\ta\le 1$, $f,g\in\C$,
\EQ{
|G'(\ta f +(1-\ta)g)| \le \mu(\ta) (H(f) + H(g)).
}
Suppose that $s\in(0,1]$, and $1<p,p_1,p_2<\I$ are such that $\rev p=\rev{p_1}+\rev{p_2}$. Then,
\EQ{
\norm{|\nabla|^sG(u)}_{L_x^p}\lsm\norm{H(u)}_{L_x^{p_1}}\norm{|\nabla|^su}_{L_x^{p_2}}.
}
\end{lem}
In this paper, we also need a technical variant of Lemma \ref{lem:Frac_chain}:
\begin{lem}[Fractional chain rule with separate inputs]\label{lem:Frac_chain-modify}
Let $G\in C^1(\C)$ be given as in Lemma \ref{lem:Frac_chain}. Suppose that $u=v+w$, $s\in(0,1]$, and $1<p,p_1,p_2,\wt p_1,\wt p_2<\I$ are such that $\rev p=\rev{p_1}+\rev{p_2} = \frac{1}{\wt p_1} + \frac{1}{\wt p_2}$. Then,
\EQ{
\norm{|\nabla|^sG(u)}_{L_x^p} \lsm \norm{H(u)}_{L_x^{p_1}}\norm{|\nabla|^sv}_{L_x^{p_2}} + \norm{H(u)}_{L_x^{\wt p_1}}\norm{|\nabla|^sw}_{L_x^{\wt p_2}}.
}
\end{lem}
\begin{proof}
First, recall the Hardy-Littlewood maximal operator $\M$:
\EQ{
	\M f(x):=\sup_{r>0}\frac{1}{|B(0,r)|}\int_{B(0,r)} |f(x-y)|\dy,
}
where $B(0,r):=\fbrk{x\in\R^d:|x|\loe r}$. 

Recall that the proof of classical result in Lemma \ref{lem:Frac_chain} is based on the observation: (see \cite{CW01JFA} for instances)
\EQn{\label{eq:Frac_chain-modify-ob1}
\norm{2^{js}P_jG(u)}_{L_j^2(\Z)} \lsm \M(H\circ u)(x)\norm{\M(2^{ks}P_ku)(x)}_{l_k^2(\Z)} + \norm{\M(2^{ks}P_ku H\circ u)(x)}_{l_k^2(\Z)},
}
then Lemma \ref{lem:Frac_chain} follows by \eqref{eq:Frac_chain-modify-ob1}, H\"older's inequality, and Lemma \ref{lem:HL-boundedness}.

Now, we turn to the proof of Lemma \ref{lem:Frac_chain-modify}. In fact, by \eqref{eq:Frac_chain-modify-ob1} and triangular inequality $|u|\le |v| + |w|$,
\EQn{\label{eq:Frac_chain-modify-ob2}
\norm{2^{js}P_jG(u)}_{L_j^2(\Z)} \lsm & \M(H\circ u)(x)\brkb{\norm{\M(2^{ks}P_kv)(x)}_{l_k^2(\Z)} + \norm{\M(2^{ks}P_kw)(x)}_{l_k^2(\Z)}} \\
& + \norm{\M(2^{ks}P_kv H\circ u)(x)}_{l_k^2(\Z)} + \norm{\M(2^{ks}P_kw H\circ u)(x)}_{l_k^2(\Z)},
}
then Lemma \ref{lem:Frac_chain-modify} follows by \eqref{eq:Frac_chain-modify-ob2}, H\"older's inequality, and Lemma \ref{lem:HL-boundedness}.
\end{proof}

\subsection{Linear Schr\"odinger equations}
Let $e^{\frac12it\De}u_0$ be the solution of 
\EQ{
	\left\{ \aligned
	&i\pd_t u + \frac12\De u = 0, \\
	& u(0,x) = u_0,
	\endaligned
	\right.
}
and we denote that $S(t):=e^{\frac12it\De}$. Then, we have the explicit formula
\EQn{\label{linear-explicit}
	S(t)u_0 = \frac{1}{(2\pi it)^{d/2}} \int_{\R^d} e^{i\frac{|x-y|^2}{2t}} u_0(y)\dy.
}
\begin{lem}[Strichartz estimate, \cite{KT98AJM,KTV14book}]\label{lem:strichartz}
Let $I\subset \R$ containing $0$. Suppose that $(q,r)$ and $(\wt{q},\wt{r})$ are $L_x^2$-admissible.
Then,
\EQn{\label{eq:strichartz-1}
\norm{ S(t)\ph}_{L_t^qL_x^r(\R\times\R^d)} \lsm \norm{\ph}_{L_x^2(\R^d)},
}
and
\EQn{\label{eq:strichartz-2}
\normb{\int_0^t S(t-s) F(s)\dd s}_{L_t^qL_x^r(I\times \R^d)} \lsm \norm{F}_{L_t^{\wt{q}'} L_x^{\wt{r}'}(I\times\R^d)}.
}
Moreover, for $2\loe q<\I$,
\EQn{\label{eq:strichartz-u2}
\norm{u}_{L_t^q L_x^r(I\times \R^3)} \lsm \norm{u}_{U_\De^q(I;L_x^2(\R^3))} \lsm \norm{u}_{U_\De^2(I;L_x^2(\R^3))},
}
and if we assume further $q\ne2$, then 
\EQn{\label{eq:strichartz-v2}
\norm{u}_{L_t^q L_x^r(I\times \R^3)}	 \lsm \norm{u}_{V_\De^2(I;L_x^2(\R^3))},
}
and
\EQn{\label{eq:strichartz-u2-dual}
\normb{\int_0^t S(t-s) F(s)\dd s}_{U_\De^2(I;L_x^2(\R^d))} \lsm \norm{F}_{L_t^{q'} L_x^{r'}(I\times\R^d)}.
}
\end{lem}
\begin{remark}
\eqref{eq:strichartz-2} and \eqref{eq:strichartz-u2-dual} also hold for $\int_{+\I}^t S(t-s) F(s)\dd s$ and $I = [T,+\I)$ with some $T>0$. In fact, for any $T_1>T$ and $F$ defining on $[T,+\I)$, applying \eqref{eq:strichartz-2} with $I=[T-T_1,0]$ and $F(t-T_1)$, then by translation in $s$ and $t$,
\EQ{
\normb{\int_{T_1}^{t} S(t-s) F(s)\dd s}_{L_t^qL_x^r([T,T_1]\times \R^d)} \lsm \norm{F}_{L_t^{\wt{q}'} L_x^{\wt{r}'}([T,T_1]\times\R^d)}.
}
Since the implicit constant in the estimate is independent of $T_1$, the desired conclusion follows by letting $T_1\ra+\I$. 
\end{remark}
\begin{lem}[Endpoint radial Strichartz estimate, \cite{GLNY18JDE}]\label{lem:radial- strichartz} Let $\ph\in\dot H_x^{1/2}(\R^3)$ be radial. Then,
\EQn{\label{eq:radial- strichartz}
\norm{ S(t)\ph}_{L_t^2L_x^\I(\R\times\R^3)} \lsm \norm{\ph}_{\dot H_x^{\frac12}(\R^3)}.
}
\end{lem}

%
%
\begin{lem}[Local smoothing, \cite{CS88JAMS,GV92CMP,KPV91Indiana}]\label{lem:local-smoothing}
We have that 
\EQ{
\sup_{j\in \Z } 2^{-\frac12j}\norm{\chi_j S(t)f}_{L_{t,x}^2(\R\times\R^3)} \lsm \normb{|\nabla|^{-\frac12}f}_{L_x^2(\R^3)}.
}
\end{lem}
\begin{remark}
In this paper, we also use the local smoothing effect estimate in the small-$|x|$ case:
\EQ{
\norm{\chi_{\le 0} S(t)f}_{L_{t,x}^2(\R\times\R^3)} \lsm \normb{|\nabla|^{-\frac12}f}_{L_x^2(\R^3)},
}
and the localized version:
\EQ{
\sup_{j\in \Z } 2^{-\frac12j}\norm{\chi_j S(t)P_Nf}_{L_{t,x}^2(\R\times\R^3)} \lsm N^{-\frac12} \norm{P_Nf}_{L_x^2(\R^3)}.
}
Moreover, these estimates can be passed into atom space $U_\De^2$, using the general extension principle (see Proposition 2.19 in \cite{HHK10Poinc}). For example, we have that 
\EQ{
\sup_{j\in \Z } 2^{-\frac12j}\norm{\chi_j P_Nu}_{L_{t,x}^2(\R\times\R^3)} \lsm N^{-\frac12} \norm{P_Nu}_{U_\De^2(\R^3)}.
}
\end{remark}
Next, we gather some properties related to the Lorentz space.
\begin{lem}
We have that the following holds:
\begin{enumerate}
\item (Equivalent formula for the Lorentz space) Let $1\le q<\I$ and $1\le p\le \I$, then
\EQn{\label{eq:lorentz-equivalent}
\norm{f}_{L^{q,p}(\R^d)} \sim_{q,p} \normb{\norm{f(\cdot)\chi_j(\cdot)}_{L^q(\R^d)}}_{l_{j\in\Z}^p}.
}
Furthermore, we have that $\norm{f}_{L^{q,q}(\R^d)} \sim \norm{f}_{L^{q}(\R^d)}$.
\item (H\"older's inequality for the Lorentz space) Let $1\le q_1,q_2<\I$ and $1\le p_1,p_2\le \I$ such that $\frac1q=\frac1{q_1}+\frac1{q_2}$ and $\frac1p=\frac1{p_1}+\frac1{p_2}$, then
\EQn{\label{eq:lorentz-holder}
\norm{fg}_{L^{q,p}(\R^d)} \lsm \norm{f}_{L^{q_1,p_1}(\R^d)} \norm{g}_{L^{q_2,p_2}(\R^d)}.
}
\item (Lorentz version of the Strichartz estimate) Let $(q,r)$ and $(\wt{q},\wt{r})$ be $L_x^2$-admissible. Then,
\EQn{\label{eq:lorentz-strichartz-1}
\norm{S(t)\ph}_{L_t^{q,2}L_x^r(\R\times\R^d)} \lsm \norm{\ph}_{L_x^2(\R^d)},
}
and
\EQn{\label{eq:lorentz-strichartz-2}
\normb{\int_0^tS(t-s) F(s)\ds}_{L_t^{q,2}L_x^r(\R\times \R^d)} \lsm \norm{F}_{L_t^{\wt{q}',2} L_x^{\wt{r}'}(\R\times\R^d)}.
}
\item (Lorentz version of Strichartz-to-$U_\De^p$ estimate) Let $(q,r)$ be $L_x^2$-admissible, $p\ge2$, and $q_1=\min\fbrk{q,p}$. Then,
\EQn{\label{eq:lorentz-strichartz-up}
\norm{u}_{L_t^{q,p} L_x^r(I\times \R^d)} \lsm \norm{u}_{U_\De^{q_1}(I;L_x^2(\R^d))} \lsm \norm{u}_{U_\De^{2}(I;L_x^2(\R^d))},
}
\item (Lorentz version of dual $U_\De^2$ estimate) Let $I\subset\R$ be an interval containing $0$. For any $1\le p<2$ and $L^2$-admissible pair $(q,r)$ with $q\ne2$,
\EQn{\label{eq:lorentz-duality}
\normb{\int_0^t S(t-s) F(s) \ds}_{U_\De^2(I;L_x^2(\R^d))}
\lsm \norm{F}_{L_t^{q'} L_x^{r'} + L_t^{q',p} L_x^{r'}(I\times\R^d)}.
}
\end{enumerate}
\end{lem}
\begin{remark}
\eqref{eq:lorentz-strichartz-up} and \eqref{eq:lorentz-duality} are new results. \eqref{eq:lorentz-strichartz-up} is quite useful when $q_1>2$, so that the related norm can be passed into $V_\De^2$. Moreover, we believe that this point would be of independent interest.
\end{remark}
\begin{proof}
\eqref{eq:lorentz-equivalent} and \eqref{eq:lorentz-holder} are basic properties of Lorentz space. The Lorentz version of Strichartz estimates \eqref{eq:lorentz-strichartz-1} and \eqref{eq:lorentz-strichartz-2} was observed by Nakanishi \cite{Nak01siam}.
	
Now, we prove \eqref{eq:lorentz-strichartz-up}, which is based on the following observation. For any sequence $\fbrk{a_k(t)}_{k\in\N_+}\subset\C$, finite partition $\mathcal{Z}:-\I<t_0<t_1<\cdots<t_K=\I$, by $l^{q_1}\hra l^q$ and Minkowski's inequality,
\EQn{\label{esti:lorentz-strichartz-up-observation}
\normb{\sum_{k=1}^K\mathbbm{1}_{[t_{k-1},t_k)}(t)a_k(t)\chi_j(t)}_{l_{j}^{p}L_t^q} \lsm \norm{a_k(t)\chi_j(t)}_{l_{j}^{p}L_t^ql_k^{q_1}} \lsm \norm{a_k(t)\chi_j(t)}_{l_k^{q_1}l_{j}^{p}L_t^q},
}
where we write $l_j^p:=l_{j}^p(\Z)$, $L_t^q:=L_t^q(\R)$, and  $l_k^{q_1}:=l_{k\in\fbrk{1,\cdots,K}}^{q_1}$ for short. Now, the desired inequality \eqref{eq:lorentz-strichartz-up} follows by standard argument. In fact, it suffices to prove that for any $U_\De^{q_1}$-atom $a$, it holds that
\EQn{\label{esti:lorentz-strichartz-up-atom}
\norm{a}_{L_t^{q,p} L_x^r(I\times \R^d)} \lsm 1.
}
Recall the definition of atom: there exists finite partition $\mathcal{Z}:-\I<t_0<t_1<\cdots<t_K=\I$ and $\fbrk{\phi_k}_{k=1}^{K}\subset L_x^2(\R^d)$ with $\sum_{k=1}^{K} \norm{\phi_k}_{L_x^2(\R^d)}^{q_1} =1$ such that
\EQ{
a=\sum_{k=1}^K \mathbbm{1}_{[t_{k-1},t_k)}(t)S(t)\phi_{k}.
}
By the equivalent formula \eqref{eq:lorentz-equivalent} for the Lorentz space, \eqref{esti:lorentz-strichartz-up-observation}, embedding $L_t^{q,2}\hra L_t^{q,p}$, and the Strichartz estimate \eqref{eq:lorentz-strichartz-1},
\EQ{
\norm{a}_{L_t^{q,p} L_x^r(\R\times \R^3)} \le & \normB{\sum_{k=1}^K \mathbbm{1}_{[t_{k-1},t_k)}(t)\norm{S(t)\phi_{k}}_{L_x^r(\R^3)}}_{L_t^{q,p} (\R)} \\ 
= & \normB{\sum_{k=1}^K \mathbbm{1}_{[t_{k-1},t_k)}(t)\norm{S(t)\phi_{k}}_{L_x^r(\R^3)}\chi_j(t)}_{l_j^pL_t^{q}(2^\Z\times\R)} \\ 
\lsm & \norm{S(t)\phi_{k}}_{l_k^{q_1}L_t^{q,p} L_x^r(\fbrk{1,\cdots,K}\times\R\times \R^3)} \\
\lsm & \norm{S(t)\phi_{k}}_{l_k^{q_1}L_t^{q,2} L_x^r(\fbrk{1,\cdots,K}\times\R\times \R^3)} \\
\lsm & \norm{\phi_{k}}_{l_k^{q_1} L_x^2(\fbrk{1,\cdots,K}\times\R^3)} \\
\lsm & 1.
}
This gives \eqref{esti:lorentz-strichartz-up-atom}, and then \eqref{eq:lorentz-strichartz-up} follows.

Finally, we prove \eqref{eq:lorentz-duality}. Recall that \eqref{eq:upvpduality-2}, it suffices to bound
\EQ{
\sup_{\norm{g}_{V_\De^2(I;L_x^2(\R^d))}=1} \abs{\int_{I}\int_{\R^d} F(t)\wb{g(t)}\dx\dd t}.
}
Note that $1\le p<2$, then its H\"older conjugate exponent satisfies $2<p'\le\I$, thus $q_1=\max\fbrk{q,p}>2$. By the Strichartz estimate \eqref{eq:lorentz-strichartz-up},
\EQn{\label{esti:lorentz-duality}
\norm{g}_{L_{t}^{q,p'} L_x^{r}(I\times\R^d)} \lsm \norm{u}_{U_\De^{q_1}(I;L_x^2(\R^d))} \lsm \norm{g}_{V_\De^2(I;L_x^2(\R^d))}.
}
Therefore, \eqref{eq:lorentz-duality} follows from  \eqref{eq:upvpduality-2}, \eqref{eq:strichartz-v2}, and \eqref{esti:lorentz-duality}.
\end{proof}

\subsection{Pseudo conformal transform}
Now, we recall some notation and properties. Define the pseudo conformal transform $\mathcal T$ by 
\EQn{\label{def:PC-transform}
	\mathcal T f(t,x) := \frac{1}{(it)^{d/2}} e^{\frac{i|x|^2}{2t}} \bar f(\frac{1}{t},\frac{x}{t}).
}
Then the inverse transform satisfies  
$$
\mathcal T^{-1}=\mathcal T.
$$

Now, we define
\EQ{
	M(t) := e^{\frac{i|x|^2}{2t}}\text{, and } J(t):=x+it\nabla.
}
Recall the explicit formula for $S(t)$ in \eqref{linear-explicit}:
\EQ{
S(t) = \frac{1}{(2\pi it)^{d/2}} \int_{\R^d} e^{i\frac{|x-y|^2}{2t}} u_0(y)\dy = \frac{1}{(2\pi it)^{d/2}} M(t) \F(M(t)u_0)(\frac xt).
}
Then, this gives that for $f(x):\R^d\ra\C$,
\EQn{\label{eq:pc-TS=SF}
\TT S(t)f=S(t)\F^{-1}\wb f,
}
and other presentations for the vector field:
\EQ{
	J(t)=S(t)xS(-t)=M(t)(it\nabla)M(-t).
}
Therefore, we define the fractional vector field for any $s>0$,
\EQ{
|J(t)|^{s}:= S(t)|x|^sS(-t).
}
Similarly, one may show that $|J(t)|^{s} = M(t)|t\nabla|^sM(-t)$. Using this formula, we have for spacetime function $u(t,x)$,
\EQn{\label{eq:pc-nablaT=SJ}
|\nabla|^s\TT u =\TT |J(t)|^s u.
}

Now, we give two basic facts for the pseudo conformal transform:
\begin{remark}
Therefore, by \eqref{eq:pc-TS=SF}, if $u$ is the solution of \eqref{eq:nls} with initial data $u_0$, then $\U=\TT u$ solves
\EQn{\label{eq:nls-pc}
	i\pd_t \U + \frac12 \De \U = t^{\frac{dp}{2}-2}|\U|^p\U,
}
with $S(-t)U$ tends to $ \F^{-1} \wb u_0$ as $t\ra+\I$ in suitable sense. Moreover, by \eqref{eq:pc-nablaT=SJ}, for any $0<a<b<\I$, we have that
\EQ{
\U \in C([a,b];\dot H_x^s(\R^d)) \quad \Longleftrightarrow & \quad |J(t)|^su \in C_t([\frac1b,\frac1a];L_x^2(\R^d)) \\
\Longleftrightarrow & \quad S(-t)u \in C_t([\frac1b,\frac1a];\F \dot H_x^s(\R^d)).
}
\end{remark}

\begin{remark}\label{rem:spacetime-exponent-transform}
	Let $u(t,x)$ be a function and
	\EQ{
		\U(s,y)=\TT u(s,y)=  \frac{1}{(is)^{d/2}} e^{\frac{i|y|^2}{2s}} \wb u(\frac{1}{s},\frac{y}{s}).
	}
	Then, for any $1\le q,r\le\I$, we have
	\EQ{
		\norm{u}_{L_t^q L_x^r(\R\times\R^d)} = \normb{s^{\frac d2-\frac2q-\frac dr}\U(s,y)}_{L_s^q L_y^r(\R\times\R^d)}.
	}
	Particularly, for any $\de>0$, $s\ge0$, and $L_x^2$-admissible $(q,r)$, we have
\EQ{
\norm{|J(t)|^su}_{L_t^q L_x^r((0,\de]\times\R^3)} = \norm{\abs{\nabla}^s\U}_{L_t^q L_x^r([1/\de,+\I)\times\R^3)}.
}
The above identity also hold for $U_\De^2(I;L_x^2(\R^d)$ and $V_\De^2(I;L_x^2(\R^d)$.
\end{remark}

\vspace{2cm}

\section{Global well-posedness}\label{sec:proof}

\vspace{0.5cm}

\subsection{High-low decomposition}\label{sec:highlow}
We apply the pseudo conformal transform, then the main theorem is reduced to the final data problem for the following nonautonomous equation:
\EQn{\label{eq:nnls}
i\pd_t \U + \frac12 \De \U = t^{-\frac{1}{2}}|\U|\U.
}
It suffices to prove the well-posedness on $[t_0,\I)$ for any given $t_0>0$.

To start with, we give the basic settings:
\begin{redu}\label{assu:main}
We assume that the following hold.
\begin{enumerate}
\item
Let $t_0>0$ be any fixed time, and let $|x|^{1/2}u_0\in L_x^2(\R^3)$, and $u_0$ be radial.
\item
Fix a sufficiently small constant $\de_0:=\de_0(t_0,u_0)>0$ that will be defined later, see \eqref{defn:delta0} below precisely. Since $|x|^{1/2}u_0\in L_x^2(\R^3)$, then there exists $N_0=N_0(\de_0,u_0)\in2^\N$ such that $N_0>1$ and
\EQn{\label{eq:initialdata-estimate-v}
	\norm{|x|^{1/2}\chi_{\goe N_0}u_0}_{L_x^2(\R^d)} \loe \de_0.
}
\end{enumerate}
\end{redu}

Take the initial data as
\EQ{
	v_0:= \chi_{\goe N_0}u_0 \text{, and }w_0:= \chi_{\loe N_0}u_0.
}
and define
\EQ{
	v:=S(t)v_0\text{, and }w:=u-v.
}
Next, we apply the pseudo-conformal transform, and define
\EQ{
	\U:=\TT u\text{, }\V:=\TT v\text{, and }\W:=\TT w.
}
Then, denote the final data as
\EQ{
	\U_+:= \F^{-1} \wb u_0\text{, } \V_+ := \F^{-1} \wb v_0\text{, and }\W_+ := \F^{-1} \wb w_0.
}
 
\begin{remark}
Throughout the proof, the norm $\norm{|x|^{1/2}u_0}_{L_x^2(\R^3)}$ can be viewed as a constant, and some parameter depending on the profile of $\U_+$ can also be regarded as constant, then we usually omit this dependence, namely write $ C = C(u_0) = C(\U_+) $ for short.
\end{remark}

Therefore, by the above settings, $w$ solves the equation
\EQ{
	\left\{ \aligned
	&i\pd_t w +\frac12\De w=|u|u, \\
	& w(0,x) = w_0.
	\endaligned
	\right.
}
We also have
\EQ{
\V = S(t)\V_+,
}
and
\EQ{
\W = S(t)\W_+ -i\int_{+\I}^tS(t-\ta)(\ta^{-\frac12}|\U|\U) \dd \ta,
}
where the integral is in the sense that
\EQ{
\lim_{t'\ra+\I}\int_{t'}^tS(t-\ta)(\ta^{-\frac12}|\U|\U) \dd \ta\quad\text{in }\dot H_x^{1/2}(\R^3).
}
By \eqref{eq:initialdata-estimate-v}, we have that
\EQn{\label{eq:finaldata-estimate-v}
\norm{\V_+}_{\dot H_x^{\frac12}(\R^3)} \le \de_0, 
}
and
\EQn{\label{eq:finaldata-estimate-w}
\norm{\W_+}_{\dot H_x^1(\R^3)} \lsm N_0^{\frac12}.
}

Combining \eqref{eq:finaldata-estimate-v}, the Strichartz estimate \eqref{eq:strichartz-1}, and the radial Strichartz estimate in Lemma \ref{lem:radial- strichartz}, we also have that
\begin{lem}[Linear estimate]
Let Reduction \ref{assu:main} hold. Then, for any $L_x^2$-admissible pair $(q,r)$,
\EQn{\label{eq:linear-estimate-v}
\normb{|\nabla|^{\frac12}\V}_{L_t^q L_x^r(\R\times\R^3)} + \normb{|\nabla|^{\frac12}\V}_{L_t^{q,2} L_x^r(\R\times\R^3)}\lsm \de_0,
}
Furthermore, since $u_0$ is radial, we also have the endpoint Strichartz estimate
\EQn{\label{eq:linear-estimate-v-radial}
\norm{\V}_{L_t^2 L_x^\I(\R\times\R^3)}\lsm \de_0.
}
\end{lem}
\begin{remark}
	Note that $\V$ is a high-frequency cut-off function. In fact, by the decomposition of initial data, we can observe that for any $t>0$, $\wh \V(t,\cdot)$ is supported on $\fbrk{\xi\in\R:|\xi|\gsm N_0}$. Therefore, we can express $\V$ as the sum:
	\EQ{
		\V=\sum_{N:N\gsm N_0} \V_N.
	}
	In the following, we will keep this point in mind, but for brevity, we will usually omit the lower bound $N_0$ when summing over $N$.
\end{remark}

\subsection{Local theory I: local well-posedness in the critical space}\label{sec:localI}
In this subsection, we prove the local well-posedness of 3D quadratic NLS in $\F \dot H^{1/2}$. This is ready included in the previous result \cite{KMMV17NoDEA}, and we give the proof for completeness. Moreover, we derive a $\dot H_x^{1/2}$-level estimate based on the $U_\De^2$-space. The local well-posedness is given as follows.
\begin{prop}\label{prop:localI}
Let Reduction \ref{assu:main} hold. Then, for sufficiently small $\de_1>0$,  there exists $T_0=T_0(\de_1,\U_+)>1$ such that the equation \eqref{eq:nls-pc} admits a unique solution 
$$\U\in C([T_0,\I); \dot H_x^\frac12(\R^d)),$$ 
such that
\EQ{
\normb{|\nabla|^{\frac12}\U}_{L_t^\I L_x^2 ([T_0,\I)\times\R^3)} \lsm 1,
}
and
\EQn{\label{eq:local-h1/2-smallness}
\normb{|\nabla|^{\frac12}\U}_{L_t^{4,2} L_x^3([T_0,\I)\times\R^3)} + \normb{\U}_{L_t^{8,2} L_x^4([T_0,\I)\times\R^3)}  \lsm \de_1.
}
Moreover, for any $L^2$-admissible $(q,r)$,
\EQn{\label{eq:local-h1/2}
\norm{\U}_{X^{\frac12}([T_0,+\I))\cap L_t^{q,2} L_x^r([T_0,+\I)\times\R^3)} \lsm 1,
}
where for any interval $I\subset\R$,
\EQn{\label{defn:xs}
\norm{\W}_{X^s(I)}:= \normb{N^s\norm{P_N\W}_{U_\De^2(I;L_x^2(\R^3))}}_{l_N^2}.
}
\end{prop}
\begin{proof}
For any interval $I\subset\R$, we define that
\EQ{
\norm{\U}_{\wt X(I)}:=  \normb{|\nabla|^{\frac12}\U}_{L_t^{4,2} L_x^3(I\times\R^3)} + \normb{\U}_{L_t^{8,2} L_x^4(I\times\R^3)} .
}
Let 
\EQ{
R:= C\norm{\U_+}_{\dot H_x^{\frac12}(\R^3)},
}
and we can assume that $R\gsm 1$, otherwise the global well-posedness already follows by small data theory. By the Strichartz estimate \eqref{eq:lorentz-strichartz-1}, for sufficiently small $\de_1>0$, there exists a large $T_0=T_0(\de_1,\U_+)>0$ such that
\EQn{\label{eq:localI-linear}
\norm{S(t)\U_+}_{\wt X([T_0,+\I))}\le \de_1.	
}
Denote that $I:=[T_0,\I)$.
Recall the Duhamel's formula:
\EQ{
	\U
	=S(t)\U_+-i\int_{+\I}^t  \ta^{-\frac12}|\U(\ta)|\U(\ta)\dd \ta =: \Psi_{\U_+}(\U).
}
Take the resolution space as
\EQ{
B_{R,\de_1} := \fbrk{\U\in C(I;\dot H_x^{\frac12}(\R^3)): \norm{\U}_{X^{\frac12}(I)}\le2R\text{, and }\norm{\U}_{\wt X(I)}\le 2\de_1}.
}
	
Now, we prove that $\Psi_{\U_+}$ maps $B_{R,\de_1}$ into itself. First, note that by H\"older's inequality for Lorentz space in $t$ and Sobolev's inequality in $x$,
\EQn{\label{eq:localI-1}
\normb{t^{-\frac12}\U}_{L_t^{\frac43,2}L_x^6(I\times\R^3)}\lsm \normb{ t^{-\frac12}}_{L_t^{2,\I}(I)} \normb{|\nabla|^{\frac12}\U}_{L_t^{4,2}L_x^3(I\times\R^3)} \lsm \norm{\U}_{\wt X(I)}.
}
Now, by the Strichartz estimates \eqref{eq:upvpduality-1} and \eqref{eq:lorentz-duality},
\EQn{\label{esti:localI-X-1}
\norm{\Psi_{\U_+}(\U)}_{X^{\frac12}(I)} \le & C\norm{\U_+}_{\dot H_x^{\frac12}(\R^3)} + C \normb{t^{-\frac12}|\nabla|^{\frac12}(|\U|\U)}_{L_t^{1} L_x^{2}(I\times\R^3)} \\
\le & R + C \normb{t^{-\frac12}|\nabla|^{\frac12}(|\U|\U)}_{L_t^{1} L_x^{2}(I\times\R^3)}.
}
By the fractional chain rule in Lemma \ref{lem:Frac_chain}, H\"older's inequality in $t$, and \eqref{eq:localI-1},
\EQn{\label{esti:localI-X-2}
\normb{t^{-\frac12}|\nabla|^{\frac12}(|\U|\U)}_{L_t^{1} L_x^{2}(I\times\R^3)} \le & C  \normb{|\nabla|^{\frac12}\U}_{L_t^{4,2} L_x^{3}(I\times\R^3)} \normb{t^{-\frac12}\U}_{L_t^{\frac43,2}L_x^6(I\times\R^3)} \\
\le & C\norm{\U}_{\wt X(I)}^2 \\
\le & C\de_1^2.
}
Noting that $0<\de_1\ll1$, by \eqref{esti:localI-X-1} and \eqref{esti:localI-X-2},
\EQn{\label{eq:localI-X}
\norm{\Psi_{\U_+}(\U)}_{X(I)} \le 2R.
}
Similarly, by \eqref{eq:localI-linear}, the Strichartz estimate \eqref{eq:lorentz-strichartz-2}, and \eqref{esti:localI-X-2},
\EQn{\label{eq:localI-Xtilde}
\norm{\Psi_{\U_+}(\U)}_{\wt X(I)} \le & \norm{S(t)\U_+}_{\wt X(I)} +
C \normb{t^{-\frac12}|\nabla|^{\frac12}(|\U|\U)}_{L_t^{1} L_x^{2}(I\times\R^3)} \\
\le & \de_1 + C\norm{\U}_{\wt X(I)}^2 \\
\le & \de_1 + C\de_1^2 \\
\le & 2\de_1.
}
\eqref{eq:localI-X} and \eqref{eq:localI-Xtilde} imply that $\Psi_{\U_+}$ maps $B_{R,\de_1}$ into itself. By similar argument, we are able to prove that $\Psi_{\U_+}$ is a contraction mapping on $B_{R,\de_1}$. This shows the local well-posedness part.

Next, we prove \eqref{eq:local-h1/2}. By the Strichartz estimates \eqref{eq:lorentz-strichartz-1}, \eqref{eq:lorentz-strichartz-2}, and \eqref{esti:localI-X-2},
\EQ{
\norm{\U}_{L_t^{q,2} L_x^r(I\times\R^3)} 
\lsm & R+  \normb{t^{-\frac12} |\nabla|^{\frac12}(|\U|\U)}_{L_t^{1}L_x^{2}(I\times\R^3)} 
\lsm  R.
}
As for the $X^{1/2}$-norm, by \eqref{eq:upvpduality-1}, \eqref{eq:lorentz-duality}, Minkowski's inequality, Lemma \ref{lem:square-function}, and \eqref{esti:localI-X-2}, 
\EQ{
\norm{\U}_{X^{1/2}(I)} \lsm & \norm{\U_+}_{\dot H_x^{\frac12}(\R^3)} +  \normB{N^{\frac12}\normb{t^{-\frac12} P_N(|\U|\U)}_{L_t^{1}L_x^{2}(I\times\R^3)}}_{l_N^2} \\
\lsm & R +  \normb{t^{-\frac12} |\nabla|^{\frac12}(|\U|\U)}_{L_t^{1}L_x^{2}(I\times\R^3)} \\
\lsm & R.
}
This completes the proof of \eqref{eq:local-h1/2}.
\end{proof}

\subsection{Key nonlinear estimates}\label{sec:nonlinear-estimate}
Now, we give some nonlinear estimates and their variants, which will be used frequently in the following.
\begin{prop}[Nonlinear estimates]\label{prop:nonlinear-estimate}
Let Reduction \ref{assu:main} hold. Suppose all the spacetime norms are restricted on $I\times\R^3$ with some bounded interval $I=[a,b]$ and $0<a<b$. Then, the following hold.
\begin{enumerate}
\item Let $\U_1\in \big(L_t^4 L_x^6 + L_t^\I \dot H_x^{1}\big)\cap L_t^\I \dot H_x^{\frac12}$, then
\EQn{\label{eq:nonl-esti-nabla1}
\normb{t^{-\frac12}\U_1\nabla\V}_{L_t^{\frac43} L_x^{\frac32} + L_t^{\frac{4}{3+6\ep}}L_x^{\frac{6}{5-6\ep}}} \lsm & a^{-\frac12} \de_0 \max\fbrk{\norm{\U_1}_{L_t^4 L_x^6}, |I|^{\frac14} \norm{\U_1}_{L_t^\I \dot H_x^1}} \\
& + a^{-\frac12} |I|^{\frac14+\frac32\ep} \de_0 \norm{\U_1}_{L_t^\I \dot H_x^{\frac12}}.
}
\item We have that
\EQn{\label{eq:nonl-esti-nabla1/2}
\normb{t^{-\frac12}|\nabla|^{\frac12}(|\U|\U)}_{L_t^{\frac43} L_x^{\frac32}} \lsm & a^{-\frac12} |I|^{\frac12}\de_0(\de_0 + \norm{\W}_{L_t^\I \dot H_x^{\frac12}}) \\
&+ a^{-\frac12} |I|^{\frac34} (\de_0 + \norm{\W}_{L_t^\I \dot H_x^{\frac12}}) \norm{\W}_{L_t^\I \dot H_x^1}.
}
\end{enumerate}
\end{prop}
\begin{proof}
We first prove \eqref{eq:nonl-esti-nabla1}. To start with, we make the spatial localization:
\EQn{\label{eq:nonl-esti-nabla1-1}
\normb{t^{-\frac12} \U_1\nabla\V}_{L_t^{\frac43} L_x^{\frac32} + L_t^{\frac{4}{3+6\ep}}L_x^{\frac{6}{5-6\ep}}} \lsm  \normb{t^{-\frac12} \chi_{\le 0} \U_1\nabla\V}_{L_t^{\frac43} L_x^{\frac32}}  + \normb{\sum_{j=1}^{+\I}t^{-\frac12} \chi_j\U_1 \wt\chi_j\nabla\V}_{ L_t^{\frac{4}{3+6\ep}}L_x^{\frac{3}{2-3\ep}}}.
}
For the first term, if $\U_1\in L_t^4 L_x^6$, by H\"older's inequality and Lemma \ref{lem:local-smoothing}, we have that
\EQn{\label{eq:nonl-esti-nabla1-2-low-1}
\normb{t^{-\frac12} \chi_{\le 0} \U_1 \nabla\V}_{L_t^{\frac43} L_x^{\frac32}} \lsm & a^{-\frac12} \norm{\U_1}_{L_t^4 L_x^6} \normb{\chi_{\le 0}\nabla\V}_{L_{t,x}^{2}} \\
\lsm & a^{-\frac12} \norm{\U_1}_{L_t^4 L_x^6} \norm{\V_+}_{\dot H_x^{\frac12}}  \\
\lsm & a^{-\frac12} \de_0 \norm{\U_1}_{L_t^4 L_x^6},
}
or if $\U_1\in  L_t^\I \dot H_x^{1}$, by Sobolev's inequality,
\EQn{\label{eq:nonl-esti-nabla1-2-low-2}
\normb{t^{-\frac12}\chi_{\le 0}\U_1\nabla\V}_{L_t^{\frac43} L_x^{\frac32}} 
\lsm & a^{-\frac12} |I|^{\frac14} \norm{\U_1}_{L_t^\I L_x^6} \normb{\chi_{\le 0}\nabla\V}_{L_{t,x}^{2}} \\
\lsm & a^{-\frac12} |I|^{\frac14} \de_0 \norm{\U_1}_{L_t^\I \dot H_x^1}.
}
For the second term, by the H\"older's inequality and Lemma \ref{lem:local-smoothing},
\EQn{\label{eq:nonl-esti-nabla1-2-high}
\normb{\sum_{j=1}^{+\I}t^{-\frac12} \chi_j \U_1 \wt\chi_j\nabla\V}_{ L_t^{\frac{4}{3+6\ep}}L_x^{\frac{3}{2-3\ep}}} \lsm & a^{-\frac12} |I|^{\frac14+\frac32\ep} \sum_{j=1}^{+\I} \normb{2^{\frac12j}\chi_j\U_1}_{L_t^{\I} L_x^{\frac{6}{1-6\ep}}} \normb{2^{-\frac12j}\wt\chi_j\nabla\V}_{ L_{t,x}^2} \\
\lsm & a^{-\frac12} |I|^{\frac14+\frac32\ep} \de_0 \sum_{j=1}^{+\I} 2^{-3\ep j} \normb{2^{(\frac12+3\ep)j}\chi_j\U_1}_{L_t^{\I} L_x^{\frac{6}{1-6\ep}}} \\
\lsm & a^{-\frac12} |I|^{\frac14+\frac32\ep} \de_0 \sup_{j=1,2,3,\cdots} \normb{2^{(\frac12+3\ep)j}\chi_j\U_1}_{L_t^{\I} L_x^{\frac{6}{1-6\ep}}}.
} 
By the radial Sobolev's inequality in Lemma \ref{lem:radial-sobolev} and \eqref{eq:linear-estimate-v},
\EQn{\label{eq:nonl-esti-nabla1-2-high-1}
	\normb{2^{(\frac12+3\ep)j}\chi_j\U_1}_{L_x^{\frac{6}{1-6\ep}}} \lsm  \normb{|x|^{\frac12+3\ep}\chi_j\U_1}_{L_x^{\frac{6}{1-6\ep}}} 
	\lsm  \normb{|x|^{\frac12+3\ep}\U_1}_{L_x^{\frac{6}{1-6\ep}}} 
	\lsm   \norm{\U_1}_{\dot H_x^{\frac12}}.
}
By \eqref{eq:nonl-esti-nabla1-2-high} and \eqref{eq:nonl-esti-nabla1-2-high-1},
\EQn{\label{eq:nonl-esti-nabla1-2-high-2}
\normb{\sum_{j=1}^{+\I}t^{-\frac12} \chi_j\U_1 \wt\chi_j\nabla\V}_{ L_t^{\frac{4}{3+6\ep}}L_x^{\frac{6}{5-6\ep}}} \lsm & a^{-\frac12} |I|^{\frac14+\frac32\ep} \de_0 \norm{\U_1}_{L_t^\I \dot H_x^{\frac12}}.
}
Therefore, \eqref{eq:nonl-esti-nabla1} follows by \eqref{eq:nonl-esti-nabla1-1}, \eqref{eq:nonl-esti-nabla1-2-low-1}, \eqref{eq:nonl-esti-nabla1-2-low-2}, and \eqref{eq:nonl-esti-nabla1-2-high-2}.

Next, we prove \eqref{eq:nonl-esti-nabla1/2}. By the fractional chain rule in Lemma \ref{lem:Frac_chain-modify}, \eqref{eq:linear-estimate-v}, and Sobolev's inequality,
\EQ{
\normb{t^{-\frac12}|\nabla|^{\frac12}(|\U|\U)}_{L_t^{\frac43} L_x^{\frac32}} 
\le &  a^{-\frac12} |I|^{\frac12} \brkb{\norm{\V}_{L_t^\I L_x^3} + \norm{\W}_{L_t^\I L_x^3} }  \normb{|\nabla|^{\frac12}\V}_{L_t^{4} L_x^3} \\
& + a^{-\frac12} |I|^{\frac34}  \brkb{\norm{\V}_{L_t^\I L_x^3} +\norm{\W}_{L_t^\I L_x^3}} \normb{|\nabla|^{\frac12}\W}_{L_t^\I L_x^3} \\
\le & a^{-\frac12} |I|^{\frac12}\de_0(\de_0 + \norm{\W}_{L_t^\I \dot H_x^{\frac12}}) \\
& + a^{-\frac12} |I|^{\frac34} (\de_0 + \norm{\W}_{L_t^\I \dot H_x^{\frac12}}) \norm{\W}_{L_t^\I \dot H_x^1}.
}
Therefore, we finish the proof of this proposition.
\end{proof}
Next, we give some variants of the nonlinear estimates in Proposition \ref{prop:nonlinear-estimate}.
\begin{cor}\label{cor:nonlinear-estimate}
Let Reduction \ref{assu:main} hold. Then, we have the following hold:
\begin{enumerate}
\item Suppose all the spacetime norms are restricted on $I\times\R^3$ with some bounded interval $I=[a,b]$ and $0<a<b$. Then, 
\EQn{\label{eq:nonl-esti-vnablav}
\normb{t^{-\frac12} \V\nabla\V}_{L_t^{\frac43} L_x^{\frac32} + L_t^{\frac{4}{3+6\ep}}L_x^{\frac{3}{2-3\ep}}} \lsm a^{-\frac12} \de_0^2 + a^{-\frac12} |I|^{\frac14+\frac32\ep} \de_0^2,
}
and
\EQn{\label{eq:nonl-esti-wnablav}
\normb{t^{-\frac12} \W\nabla\V}_{L_t^{\frac43} L_x^{\frac32} + L_t^{\frac{4}{3+6\ep}}L_x^{\frac{3}{2-3\ep}}} \lsm & a^{-\frac12} |I|^{\frac14} \de_0 \norm{\W}_{L_t^\I \dot H_x^1}  + a^{-\frac12} |I|^{\frac14+\frac32\ep} \de_0 \norm{\W}_{L_t^\I\dot H_x^{\frac12}}.
}
\item 
Suppose all the spacetime norms are restricted on $[T_0,+\I)\times\R^3$. Then, we have
\EQn{\label{eq:nonl-esti-vnablav-timeinfty}
\normb{t^{-\frac12} \V\nabla\V}_{L_t^{\frac43} L_x^{\frac32} + L_t^{\frac{4}{3+6\ep}}L_x^{\frac{3}{2-3\ep}}} \lsm T_0^{-\frac12} \de_0^2 +  T_0^{-\frac14+\frac32\ep} \de_0^2,
}
and
\EQn{\label{eq:nonl-esti-wnablav-timeinfty}
\normb{t^{-\frac12} \W\nabla\V}_{L_t^{\frac43} L_x^{\frac32} + L_t^{\frac{4}{3+6\ep}}L_x^{\frac{3}{2-3\ep}}} \lsm  T_0^{-\frac14} \de_0 \norm{\W}_{L_t^\I \dot H_x^1}  + T_0^{-\frac14+\frac32\ep} \de_0 \norm{\W}_{L_t^\I\dot H_x^{\frac12}}.
}
\item Suppose all the spacetime norms are restricted on $I\times\R^3$ with some bounded interval $I=[a,b]$ and $0<a<b$. Then,
\EQn{\label{eq:nonl-esti-vnablav-frequency}
& \normb{t^{-\frac12} \V\norm{N\V_N}_{l_N^2}}_{L_t^{\frac43} L_x^{\frac32} + L_t^{\frac{4}{3+6\ep}}L_x^{\frac{3}{2-3\ep}}} + \normb{t^{-\frac12} \sup_{N_1\in 2^\Z}|\V_{N_1}|\norm{N\V_N}_{l_N^2}}_{L_t^{\frac43} L_x^{\frac32} + L_t^{\frac{4}{3+6\ep}}L_x^{\frac{3}{2-3\ep}}} \\
\lsm & a^{-\frac12} \de_0^2 + a^{-\frac12} |I|^{\frac14+\frac32\ep} \de_0^2.
}
\end{enumerate} 
\end{cor}
\begin{proof}
\eqref{eq:nonl-esti-vnablav} and \eqref{eq:nonl-esti-wnablav} follow directly from \eqref{eq:nonl-esti-nabla1} by taking $\U_1=\V$ or $\U_1=\W$. Note that when applying \eqref{eq:nonl-esti-nabla1}, we always choose $L_t^4 L_x^6$-norm if $\U_1=\V$, and choose $L_t^\I \dot H_x^1$-norm if $\U_1=\W$. \eqref{eq:nonl-esti-vnablav-timeinfty} and \eqref{eq:nonl-esti-wnablav-timeinfty} also follow by slightly modifying the time integral when we apply the proof of \eqref{eq:nonl-esti-nabla1}.

\eqref{eq:nonl-esti-vnablav-frequency} follows from the following two observations. The first is that when we replace $\nabla\V$ by $\norm{N\V_N}_{l_N^2}$, then the local smoothing estimate also holds:
\EQ{
\normb{2^{-\frac12j}\wt\chi_j\norm{N\V_N}_{l_N^2}}_{ L_{t,x}^2} \lsm & \normb{2^{-\frac12j}\wt\chi_jN\V_N}_{l_N^2L_{t,x}^2} \lsm \normb{N^{\frac12}P_N\V_+}_{l_N^2L_{x}^2} \lsm \norm{\V_+}_{\dot H_x^{\frac12}}.
}
The related estimate for $\normb{\chi_{\le 0}\norm{N\V_N}_{l_N^2}}_{ L_{t,x}^2}$ also follows similarly. The second is that when we replace $\V$ by $\sup_{N_1\in 2^\Z} |\V_{\le N_1}|$, we can apply Lemma \ref{lem:linfty-littewoodpaley} to obtain
\EQ{
\normb{\sup_{N_1\in 2^\Z} |\V_{\le N_1}|}_{L_t^4 L_x^6} \lsm \norm{\V}_{L_t^4 L_x^6},
}
and
\EQ{
\normb{|x|^{\frac12+3\ep}\sup_{N_1\in 2^\Z} |\V_{\le N_1}|}_{L_x^{\frac{6}{1-6\ep}}} 
\lsm  \normb{|x|^{\frac12+3\ep}\V}_{L_x^{\frac{6}{1-6\ep}}}.
}
Combining the above observations and the proof of \eqref{eq:nonl-esti-vnablav}, we can prove \eqref{eq:nonl-esti-vnablav-frequency}.
\end{proof}

\subsection{Local theory II: energy estimate near the infinity time}\label{sec:localII}
We have given that $\U$ has $\dot H_x^{1/2}$-estimates in the neighborhood of the infinite time. In this subsection, we improve the nonlinear remainder $\W$ to $\dot H_x^1$-level. 
The main result in this subsection is as follows.
\begin{prop}\label{prop:localII}
Let Reduction \ref{assu:main} hold and $T_0$ be given in Proposition  \ref{prop:localI}. Then, we have that $\W\in C([T_0,+\I); \dot H_x^1(\R^3))$ such that
\EQ{
\norm{\W}_{X^1([T_0,+\I))} \lsm N_0^{\frac12},
}
where $X^1$-norm is defined in \eqref{defn:xs}.
\end{prop}
\begin{lem}[Nonlinear estimate]\label{lem:localII}
Let Reduction \ref{assu:main} hold and $T_0$ be given in Proposition  \ref{prop:localI}. Then, for any $T_0'\ge T_0$,
\EQn{\label{eq:localII}
\normb{\int_{+\I}^t S(t-\ta)(\ta^{-\frac{1}{2}}|\U|\U) \dd \ta}_{X^1([T_0',+\I))} \lsm & (\de_1 + \de_0) \norm{\W}_{X^1([T_0',+\I))}  + \de_0.
}
\end{lem}
Now, we prove that Lemma \ref{lem:localII} implies Proposition \ref{prop:localII}:
\begin{proof}[Proof of Proposition \ref{prop:localII}]
Let $T_0$ be given in in Proposition \ref{prop:localI} and $I:=[T_0,+\I)$. 
By \eqref{eq:finaldata-estimate-w}, let $R_1$ be defined by
\EQ{
	\norm{S(t)\W_+}_{L_t^\I \dot H_x^1(\R\times\R^3)} \le C \norm{\W_+}_{\dot H_x^1(\R^3)} \le C N_0^{1-s}=:R_1.
}
Recall the space of solutions used in Proposition \ref{prop:localI}:
\EQ{
B_{R,\de_1} := \fbrk{\U\in C(I;\dot H_x^{\frac12}(\R^3)): \norm{\U}_{X(I)}\le2R\text{, and }\norm{\U}_{\wt X(I)}\le 2\de_1},
}
where $R$ and $\de_1$ are given in Proposition \ref{prop:localI}. Now, we take the resolution space $B_{R_1,R,\de_1}$:
\EQ{
B_{R_1,R,\de_1} := \Big\lbrace\W\in C(I; H_x^1(\R^3)) : & \norm{\W}_{X^1(I)}\le 2R_1, \W+\V\in B_{R,\de_1}\Big\rbrace.
}
Define
\EQ{
\Phi_{\W_+}(\W) := 
S(t)\W_+-i\int_\I^t  \ta^{-\frac12}|(\W+\V)(\ta)|(\W+\V)(\ta)\dd \ta.
}
Note that $\Phi_{\W_+}(\W) + \V=\Psi_{\U_+}(\U)$, where $\Psi_{\U_+}$ is defined in Proposition \ref{prop:localI}:
\EQ{
\Psi_{\U_+}(\U) := S(t)\U_+-i\int_{+\I}^t  \ta^{-\frac12}|\U(\ta)|\U(\ta)\dd \ta.
}
Next, we will prove that $\Phi_{\W_+}$ is a contraction mapping on $B_{R_1,R,\de_1}$. 

First, we prove that $\Phi_{\W_+}$ maps $B_{R_1,R,\de_1}$ into itself. If $\W\in B_{R_1,R,\de_1}$, then $\U=\W+\V\in B_{R,\de_1}$. By Proposition \ref{prop:localI}, $\Psi_{\U_+}(B_{R,\de_1}) \subset B_{R,\de_1}$. Then, $\Phi_{\W_+}(\W) + \V = \Psi_{\U_+}(\U) \in B_{R,\de_1}$. Next, we prove $\norm{\Phi_{\W_+}(\W)}_{L_t^\I \dot H_x^1(I\times\R^3)}\le 2R_1$. By Lemma \ref{lem:localII},
\EQ{
\norm{\Phi_{\W_+}(\W)}_{X^1(I)} \le & C\norm{\W_+}_{\dot H_x^1(\R^3)} + C \normb{\int_t^{+\I} S(t-\ta)(\ta^{-\frac{1}{2}}|\U|\U) \dd \ta}_{X^1(I)} \\
\le & R_1 + C(\de_1 + \de_0) \norm{\W}_{X^1(I)} +\de_0 \\
\le & \frac32R_1 + C(\de_1 +\de_0) R_1,
}
for some $C>1$. Note that $\de_0$ and $\de_1$ are sufficiently small, then
\EQn{\label{esti:localII-w-3d-smallness}
\de_0\le\frac12R_1;\quad C(\de_1 + \de_0) <\frac12.
}
Therefore,
\EQ{
	\norm{\Phi_{\W_+}(\W)}_{L_t^\I \dot H_x^1([T_0,+\I)\times\R^3)} \le & 2R_1.
}
This gives that $\Phi_{\W_+}(B_{R_1,R,\de_1}) \subset B_{R_1,R,\de_1}$.

Next, we prove that $\Phi_{\W_+}$ is a contraction mapping on $B_{R_1,R,\de_1}$. By slightly modifying the argument of Lemma \ref{lem:localII}, for $\W_1,\W_2\subset B_{R_1,R,\de_1}$,
\EQ{
& \normb{\int_t^{+\I} S(t-\ta)\ta^{-\frac{1}{2}}(|\W_1 + \V|(\W_1 + \V) - |\W_2 + \V|(\W_2 + \V)) \dd \ta}_{X^1(I)} \\
\lsm & C(\de_1 + \de_0) \norm{\W_1-\W_2}_{X^1(I)}.
}
Then, by the similar argument as above, we have that
\EQ{
	\norm{\Phi_{\W_+}(\W_1) - \Phi_{\W_+}(\W_2)}_{X^1(I)} \le \frac12 \norm{\W_1-\W_2}_{X^1(I)}.
}
This gives that $\Phi_{\W_+}$ is a contraction mapping on $B_{R_1,R,\de_1}$. 

Now, there exists a unique solution to $\W=\Phi_{\W_+}(\W)$ in $B_{R_1,R,\de_1}$. Note also that for any $\W\in B_{R_1,R,\de_1}$, we have that $\W+\V\in B_{R,\de_1}$, then $\U=\V+\W$ coincides with the local solution derived in Proposition \ref{prop:localI}. This gives Proposition \ref{prop:localII}.
\end{proof}
\begin{proof}[Proof of Lemma \ref{lem:localII}]
In the proof of this lemma, we restrict the variable $(t,x)$ on $[T_0',+\I)\times\R^3$.
By the Strichartz estimate \eqref{eq:lorentz-duality} and the pointwise bound $|\nabla(|\U|\U)| \lsm |\U\nabla\U|$,
\EQnnsub{
\normb{\int_{+\I}^t S(t-\ta) \ta^{-\frac12}|\U(\ta)|\U(\ta)\dd \ta}_{X^1} 
\lsm & \normb{ \nabla( t^{-\frac12}|\U|\U)}_{N}  \nonumber\\
\lsm & \normb{  t^{-\frac12}\U\nabla\W}_{L_t^{\frac43,\frac85} L_x^{\frac32}} \label{eq:localII-nablaw}\\
&+ \normb{t^{-\frac12} \V\nabla\V}_{L_t^{\frac43} L_x^{\frac32} + L_t^{\frac{4}{3+6\ep}}L_x^{\frac{3}{2-3\ep}}} \label{eq:localII-nablav-v}\\
& + \normb{ t^{-\frac12}\W\nabla\V}_{L_t^{\frac43} L_x^{\frac32} + L_t^{\frac{4}{3+6\ep}}L_x^{\frac{3}{2-3\ep}}}, \label{eq:localII-nablav-w}
}
where $N:= L_t^{\frac43,\frac85} L_x^{\frac32} + L_t^{\frac43} L_x^{\frac32} + L_t^{\frac{4}{3+6\ep}}L_x^{\frac{3}{2-3\ep}}$. First, we deal with the term \eqref{eq:localII-nablaw}. By H\"older's inequality \eqref{eq:lorentz-holder},  \eqref{eq:local-h1/2-smallness}, and \eqref{eq:linear-estimate-v},
\EQn{\label{esti:localII-nablaw}
\eqref{eq:localII-nablaw} \lsm & \normb{t^{-\frac12}}_{L_t^{2,\I}} \norm{\U}_{L_t^{8,2} L_x^4} \norm{\nabla \W}_{L_t^{8,8} L_x^{\frac{12}{5}}} \\
\lsm & \de_1 T_0^{-1} \norm{\nabla \W}_{L_t^{8} L_x^{\frac{12}{5}}} \\
\lsm & \de_1 T_0^{-1} \norm{\W}_{X^1}.
}
Note that in the last inequality in \eqref{esti:localII-nablaw}, we use Lemma \ref{lem:square-function}, Minkowski's inequality, and \eqref{eq:strichartz-u2},
\EQ{
\norm{\nabla \W}_{L_t^{8} L_x^{\frac{12}{5}}} \lsm \norm{ N \W_N}_{l_N^2L_t^{8} L_x^{\frac{12}{5}}} \lsm \norm{N\W_N}_{l_N^2U_\De^2} = \norm{\W}_{X^1}.
}
In the following, we will omit this procedure for short. Next, by \eqref{eq:nonl-esti-vnablav-timeinfty}, \eqref{eq:nonl-esti-wnablav-timeinfty}, and \eqref{eq:local-h1/2}, we have
\EQn{\label{esti:localII-nablav-v} 
\eqref{eq:localII-nablav-v}  
\lsm T_0^{-\frac12} \de_0^2 +  T_0^{-\frac14+\frac32\ep} \de_0^2 \lsm T_0^{-\frac14+} \de_0^2,
}
and further using \eqref{eq:linear-estimate-v},
\EQn{\label{esti:localII-nablav-w}
\eqref{eq:localII-nablav-w} \lsm T_0^{-\frac14} \de_0 \norm{\W}_{X^1} + T_0^{-\frac14+\frac32\ep} \de_0.
}
Therefore, noting that $T_0>1$, \eqref{eq:localII} follows by \eqref{esti:localII-nablaw}, \eqref{esti:localII-nablav-v}, and \eqref{esti:localII-nablav-w}.
\end{proof}
Therefore, by Proposition \ref{prop:localII},
\EQn{\label{esti:energy-estimate-t0-h1}
\norm{\W(T_0)}_{\dot H_x^1(\R^3)} \lsm \norm{\W}_{L_t^\I\dot H_x^1([T_0,+\I)\times\R^3)} \lsm \norm{\W}_{X^1([T_0,+\I))} \lsm N_0^{\frac12}.
}
Moreover, by Sobolev's inequality and \eqref{eq:local-h1/2},
\EQn{\label{esti:energy-estimate-t0-l3}
\norm{\W(T_0)}_{L_x^3(\R^3)} \lsm \norm{\W}_{L_t^\I \dot H_x^{\frac12}([T_0,+\I)\times\R^3)} \lsm \norm{\W}_{X^{\frac12}([T_0,+\I))} \lsm 1.
}
Define the pseudo conformal energy:
\EQn{\label{defn:pseudo-conformal-energy}
\E(t) := \frac{1}{4}t^{\frac{1}{2}}\int_{\R^3}|\nabla\W(t,x)|^2 \dx + \frac{1}{3} \int_{\R^3}|\W(t,x)|^{3} \dx.
}
Therefore, by \eqref{esti:energy-estimate-t0-h1} and \eqref{esti:energy-estimate-t0-l3}, we have that:
\begin{cor}\label{cor:energy-estimate-t0}
Let Reduction \ref{assu:main} hold and $T_0$ be given in Proposition  \ref{prop:localI}. Then, 
\EQ{
\E(T_0)\le A N_0,
}
for some constants $A:=A(\norm{\U_+}_{\dot H_x^{\frac12}(\R^3)},T_0)>0$.
\end{cor}

\subsection{Local theory III: extending the solution}\label{sec:localIII}
Now, we are going to extend the solution from $T_0$ to $t_0$, for any fixed $t_0>0$. This will show the global well-posedness.
Next, rewrite the integral equation of $\U$ as
\EQ{
\U = & S(t)\V_+ + S(t)\W_+ - i\int_{+\I}^t  \ta^{-\frac12}|\U(\ta)|\U(\ta)\dd \ta \\
= & S(t)\V_+ +S(t-t_1)\W(t_1)-i\int_{t_1}^t  \ta^{-\frac12}|\U(\ta)|\U(\ta)\dd \ta.
}
Therefore, we consider the Cauchy problem of the equation
\EQn{\label{eq:nls-w-localIII}
\W(t)=S(t-t_1)\W(t_1)-i\int_{t_1}^t  \ta^{-\frac12}|\U(\ta)|\U(\ta)\dd \ta.
}
By modifying the argument of the above lemma, we have that
\begin{prop}\label{prop:localIII}
Let Reduction \ref{assu:main} hold, and $T_0$ be given in Proposition \ref{prop:localI}. Then, for any $t_1\in (t_0,T_0]$, there exists $0<t_2\le\min\fbrk{1,t_0^{10},t_1-t_0}$ that depends on  $\norm{\W(t_1)}_{\dot H_x^{\frac12}\cap \dot H_x^1(\R^3)}$ and $t_0$ such that the equation \eqref{eq:nls-w-localIII} admits a unique solution $\W \in C([t_1-t_2,t_1];\dot H_x^{\frac12}\cap\dot H_x^1(\R^3))$.
\end{prop}
\begin{proof}
Take some $t_2$ that will be defined later such that $0<t_2\le\min\fbrk{1,t_0^{10},t_1-t_0}$. In the proof of this proposition, we denote $I:=[t_1-t_2,t_1]$, and restrict the variable $(t,x)$ on $I\times\R^3$. Define the solution mapping:
\EQ{
\Phi_{\W(t_1)}(\W) := S(t-t_1)\W(t_1)-i\int_{t_1}^t  \ta^{-\frac12}|\U(\ta)|\U(\ta)\dd \ta.
}
Recall the definition of $X^s$ in \eqref{defn:xs}. By \eqref{eq:upvpduality-1},
\EQ{
\norm{S(t-t_1)\W(t_1)}_{X^{\frac12} \cap X^{1}(\R)} \le C\norm{\W(t_1)}_{\dot H_x^{\frac12}(\R^3)} + C\norm{\W(t_1)}_{\dot H_x^1(\R^3)}=:R.
}
Assume without loss of generality that $R\gsm 1$, otherwise the global well-posedness follows easily by small data theory. Take the resolution space as
\EQ{
B_{R} := \fbrk{\U\in C(I; \dot H_x^{\frac12} \cap \dot H_x^{1}(\R^3)): \norm{\U}_{X^{\frac12}(I)}\le 2R,\  \norm{\U}_{X^1(I)}\le 2R}.
}

Now, we prove that $\Phi_{\W(t_1)}$ maps $B_{R}$ into itself. First, we prove the $X^{\frac12}$-bound. Note that $t>t_0>t_2^{\frac{1}{10}}$, by the Strichartz estimates \eqref{eq:upvpduality-1} and \eqref{eq:strichartz-u2-dual}, \eqref{eq:nonl-esti-nabla1/2},  \eqref{eq:linear-estimate-v}, and Young's inequality,
\EQn{\label{esti:prop:localIII-h1/2}
\norm{\Phi_{\W(t_1)}(\W)}_{X^{\frac12}} \le & C\norm{\W(t_1)}_{\dot H_x^{\frac12}} + C\normb{\int_{t_1}^t S(t-\ta)  \ta^{-\frac12}|\U(\ta)|\U(\ta)\dd \ta}_{X^{\frac12}} \\
\le & R +  \normb{t^{-\frac12}|\nabla|^{\frac12}(|\U|\U)}_{L_t^1 L_x^2} \\
\le & R + Ct_0^{-\frac12} t_2^{\frac12}\de_0(\de_0 + \norm{\W}_{X^{\frac12}}) 
+ Ct_0^{-\frac12} t_2^{\frac34} (\de_0 + \norm{\W}_{X^{\frac12}}) \norm{\W}_{X^{1}} \\
\le & R + C t_2^{\frac{9}{20}}(\de_0^2 + \norm{\W}_{X^{\frac12}}^2 + \norm{\W}_{X^{1}}^2).
}
Next, we prove the $X^{1}$-bound.
By the Strichartz estimates \eqref{eq:upvpduality-1} and \eqref{eq:strichartz-u2-dual},
\EQn{\label{eq:localIII-1}
\norm{\Phi_{\W(t_1)}(\W)}_{X^1} \lsm & \norm{\W(t_1)}_{\dot H_x^1} + \normb{\int_{t_1}^t S(t-\ta)  \ta^{-\frac12}|\U(\ta)|\U(\ta)\dd \ta}_{X^1},
}
then it suffices to consider the integral term.
\EQnnsub{
\normb{\int_{t_1}^t S(t-\ta)  \ta^{-\frac12}|\U(\ta)|\U(\ta)\dd \ta}_{X^1} \lsm & \normb{  t^{-\frac12}\U\nabla\W}_{L_t^{1} L_x^{2}} \label{eq:localIII-nablaw}\\
& + \normb{  t^{-\frac12}\V\nabla\V}_{L_t^{\frac43} L_x^{\frac32} + L_t^{\frac{4}{3+6\ep}}L_x^{\frac{3}{2-3\ep}}} \label{eq:localIII-nablav-v}\\
& + \normb{  t^{-\frac12}\W\nabla\V}_{L_t^{\frac43} L_x^{\frac32} + L_t^{\frac{4}{3+6\ep}}L_x^{\frac{3}{2-3\ep}}}. \label{eq:localIII-nablav-w}
}
First, note that $t>t_0>t_2^{\frac{1}{10}}$, H\"older's inequality, \eqref{eq:strichartz-u2}, and \eqref{eq:linear-estimate-v},
\EQn{\label{esti:localIII-nablaw}
\eqref{eq:localIII-nablaw} \lsm & t_0^{-\frac12} \brkb{t_2^{\frac34} \norm{\W}_{L_t^{\I} L_x^{6}} \norm{\nabla \W}_{L_t^{4} L_x^{3}} + t_2^{\frac12} \norm{\V}_{L_t^{4} L_x^{6}} \norm{\nabla \W}_{L_t^{4} L_x^{3}} }  \\
\lsm & t_2^{-\frac1{20}} \brkb{t_2^{\frac34} \norm{\W}_{X^1}^2 + t_2^{\frac12} \de_0 \norm{\W}_{X^1} }  \\
\lsm & t_2^{\frac{7}{10}} \norm{\W}_{X^1}^2 + \de_0 \norm{\W}_{X^1}.
}
For \eqref{eq:localIII-nablav-v}, using \eqref{eq:nonl-esti-vnablav} and $t_2\le 1$, we have that 
\EQn{\label{esti:localIII-nablav-v}
\eqref{eq:localIII-nablav-v} \lsm & t_0^{-\frac12} \de_0^2 + t_0^{-\frac12} t_2^{\frac14+\frac32\ep} \de_0^2 \\
\lsm & t_0^{-\frac12} \de_0^2.
}
For \eqref{eq:localIII-nablav-w}, by \eqref{eq:nonl-esti-wnablav}, we have that
\EQn{\label{esti:localIII-nablav-w}
\eqref{eq:localIII-nablav-w} 
\lsm & t_0^{-\frac12} t_2^{\frac14} \de_0 \norm{\W}_{X^1}  +  t_0^{-\frac12} t_2^{\frac14+\frac32\ep} \de_0 \norm{\W}_{X^{\frac12}} \\
\lsm & t_2^{\frac15} \de_0 (\norm{\W}_{X^{\frac12}} + \norm{\W}_{X^1}).
}
Now, combining \eqref{eq:localIII-1}, \eqref{esti:localIII-nablaw},  \eqref{esti:localIII-nablav-v}, and \eqref{esti:localIII-nablav-w},
\EQn{\label{esti:prop:localIII-h1}
\norm{\Phi_{\W(t_1)}(\W)}_{X^1} \le & R + Ct_2^{\frac{7}{10}} \norm{\W}_{X^1}^2 + C\de_0 \norm{\W}_{X^1} + Ct_0^{-\frac12} \de_0^2 + t_2^{\frac15} \de_0 (\norm{\W}_{X^{\frac12}} + \norm{\W}_{X^1}) \\
\le & R + Ct_2^{\frac{7}{10}} \norm{\W}_{X^1}^2 + C\de_0 \norm{\W}_{X^1} + C\de_0 \norm{\W}_{X^{\frac12}} + Ct_0^{-\frac12}\de_0^2.
}
Now, take $t_2$ such that
\EQ{
t_2:= \min\fbrk{1,t_0^{10},t_1-t_0,\brkb{\frac1{4CR}}^{\frac{20}{9}},\brkb{\frac1{16CR}}^{\frac{10}{7}}},
}
and take sufficiently small $\de_0 = \de_0(t_0,\U_+)$ temporarily such that
\EQn{\label{eq:delta0}
\de_0 \le \brkb{ \frac{t_0^{\frac12}R}{16C}}^{\frac12};\quad C\de_0 \le \frac1{16}.
}
Therefore, by \eqref{esti:prop:localIII-h1/2} and \eqref{esti:prop:localIII-h1},
\EQ{
\norm{\Phi_{\W(t_1)}(\W)}_{X^{\frac12}}
\le 2R;\quad\norm{\Phi_{\W(t_1)}(\W)}_{X^1}
\le 2R.
}
This proves that $\Phi_{\W(t_1)}$ maps $B_{R}$ into itself. By modifying the above argument, we can also prove that $\Phi_{\W(t_1)}$ is a contraction mapping on $B_{R}$, which gives this proposition.
\end{proof}
\subsection{Energy estimate towards the origin}\label{sec:energy}
\begin{prop}[A priori estimate]\label{prop:energy}
Let Reduction \ref{assu:main} hold and $T_0$ be given in Proposition \ref{prop:localI}. Recall that there exists $A>0$ such that
\EQn{\label{eq:energy-assumption}
\E(T_0) \le A N_0,
}
where $\E$ is defined in \eqref{defn:pseudo-conformal-energy}, and $A$ is given in Corollary \ref{cor:energy-estimate-t0}. Take any $t_3\in[t_0,T_0)$ such that $\W\in C([t_3,T_0];\dot H_x^1(\R^3))$. Then, we have
\EQ{
\sup_{t\in[t_3,T_0]}\E(t)\loe 2A N_0.
}
\end{prop}
Let $I=[t_3,T_0]$. To prove this  proposition, we implement a bootstrap procedure on $I$: assume an a priori bound 
\EQn{\label{eq:bound-w-hypothesis}
	\sup_{t\in I} \E(t)\loe 2AN_0,
} 
then it suffices to prove that
\EQn{\label{eq:bound-w-conclusion}
	\sup_{t\in I} \E(t)\loe \frac32AN_0.
}
In the proof of this proposition, we omit the dependence on $A$ temporarily before the final step, namely write $C=C(A)$ for short. All the space-time norms are taken over $I\times\R^3$. We also assume without loss of generality that $0< t_0 \le 1$ and $T_0\ge 1$. In the following, we will frequently use $t_3^{\al} \le t_0^{\al}$ for any $\al<0$.

Before starting the proof, we first give some useful spacetime bounds:
\begin{lem}[Various spacetime estimates]\label{lem:spacetime}
All the space-time norms are taken over $I\times\R^3$ where $I:=[t_3,T_0]$. Under the same assumptions as in Proposition \ref{prop:energy}, and letting the bootstrap hypothesis \eqref{eq:bound-w-hypothesis} holds, we have that
\begin{enumerate}
\item ($\dot H_x^1$-estimate)
\EQn{\label{eq:energy-w-h1}
	\norm{\nabla \W}_{L_t^\I L_x^2} \lsm t_0^{-\frac14} N_0^{\frac12}.
}
\item ($L_x^3$-estimate)
\EQn{\label{eq:energy-w-l3}
	\norm{\W}_{L_t^\I L_x^3} \lsm N_0^{\frac13}.
}
\item ($\norm{|x|^{1/2}\cdot}_{L_x^\I}$-estimate)
\EQn{\label{eq:energy-w-linfty}
\normb{|x|^{\frac12}\W}_{L_{t,x}^\I} \lsm t_0^{-\frac14} N_0^{\frac12}.
}
\item ($U_\De^2$-estimate)
For any $\frac12\le s\le 1$,
\EQn{\label{eq:energy-w-xs}
	\norm{\W}_{X^s} \lsm t_0^{-2} T_0^2 N_0^{\frac23+\frac13s},
}
where the $X^s$-norm is defined in \eqref{defn:xs}.
\end{enumerate}
\end{lem}
\begin{proof}
\eqref{eq:energy-w-h1} and \eqref{eq:energy-w-l3} follow directly from \eqref{eq:bound-w-hypothesis}. \eqref{eq:energy-w-h1} together with Lemma \ref{lem:radial-sobolev} imply \eqref{eq:energy-w-linfty}.

Now, we prove \eqref{eq:energy-w-xs}. First, we give the global bound of $\norm{\W}_{X^{1/2}}$. Recall the solution of $\W$ in \eqref{eq:nls-w-localIII} with $t_1=T_0$:
\EQ{
	\W(t)=S(t-T_0)\W(T_0)-i\int_{T_0}^t  \ta^{-\frac12}|\U(\ta)|\U(\ta)\dd \ta.
}
Note that $t_3\ge t_0$. By the Strichartz estimates \eqref{eq:upvpduality-1} and \eqref{eq:strichartz-u2-dual}, Sobolev's inequality, \eqref{eq:nonl-esti-nabla1/2}, \eqref{eq:local-h1/2}, \eqref{eq:energy-w-h1}, and \eqref{eq:energy-w-l3},
\EQn{\label{eq:energy-w-x1/2}
\norm{\W}_{X^{\frac12}} \lsm & \norm{\W(T_0)}_{\dot H_x^{\frac12}} + \normb{t^{-\frac12}|\nabla|^{\frac12}(|\U|\U)}_{L_t^{\frac43} L_x^{\frac32}} \\
\lsm & 1 + t_0^{-\frac12} |I|^{\frac12} \brkb{\de_0 + \norm{\W}_{L_t^\I L_x^3} }  \de_0 \\
& + t_0^{-\frac12}|I|^{\frac34} \brkb{\de_0 +\norm{\W}_{L_t^\I L_x^3}} \normb{|\nabla|^{\frac12}\W}_{L_t^\I L_x^3} \\
\lsm & 1 + t_0^{-\frac12} T_0^{\frac12} (\de_0 + N_0^{\frac13}) \de_0 + t_0^{-1} T_0^{\frac34} (\de_0 + N_0^{\frac13}) N_0^{\frac12} \\
\lsm & t_0^{-1} T_0^{\frac34} N_0^{\frac56}.
}

Next, we give the $X^1$-bound. Similar to \eqref{eq:energy-w-x1/2}, by the Strichartz estimates \eqref{eq:upvpduality-1} and \eqref{eq:strichartz-u2-dual},
\EQn{\label{eq:energy-w-x1-1}
\norm{\W}_{X^{1}} \lsm & \norm{\W(T_0)}_{\dot H_x^{1}} + \normb{t^{-\frac12}\nabla(|\U|\U)}_{L_t^{\frac43} L_x^{\frac32} + L_t^{\frac{4}{3+6\ep}}L_x^{\frac{6}{5-6\ep}}} \\
\lsm & \norm{\W(T_0)}_{\dot H_x^{1}} + t_0^{-\frac12} \norm{\W\nabla\W}_{L_t^{\frac43} L_x^{\frac32}} + t_0^{-\frac12} \norm{\V\nabla\W}_{L_t^{\frac43} L_x^{\frac32}} \\
& +  \normb{t^{-\frac12}\V\nabla\V}_{L_t^{2} L_x^{\frac65} + L_t^{\frac{4}{3+6\ep}}L_x^{\frac{6}{5-6\ep}}} +  \normb{t^{-\frac12}\W\nabla\V}_{L_t^{\frac43} L_x^{\frac32} + L_t^{\frac{4}{3+6\ep}}L_x^{\frac{6}{5-6\ep}}}.
}
By \eqref{eq:energy-w-h1}, \eqref{eq:linear-estimate-v}, and H\"older's and Sobolev's inequalities, we have
\EQn{\label{eq:energy-w-x1-2-1}
\norm{\W\nabla\W}_{L_t^{\frac43} L_x^{\frac32}} \lsm |I|^{\frac34} \norm{\W}_{L_t^\I L_x^{6}} \norm{\nabla\W}_{L_t^\I L_x^2} \lsm t_0^{-\frac12} T_0^{\frac34} N_0,
}
and
\EQn{\label{eq:energy-w-x1-2-2}
\norm{\V\nabla\W}_{L_t^{\frac43} L_x^{\frac32}} \lsm |I|^{\frac12} \norm{\V}_{L_t^4 L_x^6} \norm{\nabla\W}_{L_t^\I L_x^2} \lsm t_0^{-\frac14} \de_0 T_0^{\frac12} N_0^{\frac12}.
}
Using \eqref{eq:nonl-esti-vnablav}, \eqref{eq:nonl-esti-wnablav}, \eqref{eq:energy-w-h1}, and \eqref{eq:energy-w-x1/2}, we have that
\EQn{\label{eq:energy-w-x1-2-3}
\normb{t^{-\frac12}\V\nabla\V}_{L_t^{2} L_x^{\frac65} + L_t^{\frac{4}{3+6\ep}}L_x^{\frac{6}{5-6\ep}}} \lsm t_0^{-\frac12} \de_0^2 + t_0^{-\frac12} T_0^{\frac14+\frac32\ep} \de_0^2 \lsm t_0^{-\frac12} T_0^{\frac12}\de_0^2,
}
and that
\EQn{\label{eq:energy-w-x1-2-4}
\normb{t^{-\frac12}\W\nabla\V}_{L_t^{\frac43} L_x^{\frac32} + L_t^1 L_x^2} \lsm & t_0^{-\frac12} T_0^{\frac14} \de_0 \norm{\W}_{L_t^\I \dot H_x^1}  + t_0^{-\frac12} T_0^{\frac14+\frac32\ep} \de_0 \norm{\W}_{L_t^\I\dot H_x^{\frac12}} \\
\lsm & t_0^{-\frac34} T_0^{\frac14} \de_0 N_0^{\frac12} + T_0^{\frac12} \de_0 \cdot t_0^{-1} T_0^{\frac34} N_0^{\frac56} \\
\lsm & t_0^{-1} T_0^{\frac54}  N_0^{\frac56}.
}
Combining \eqref{eq:energy-w-x1-1}--\eqref{eq:energy-w-x1-2-4}, 
\EQn{\label{eq:energy-w-x1}
\norm{\W}_{X^{1}} \lsm & t_0^{-2} T_0^{2} N_0.
}
Then, \eqref{eq:energy-w-xs} follows by interpolation between \eqref{eq:energy-w-x1/2} and \eqref{eq:energy-w-x1}.
\end{proof}

The spacetime bounds above are very helpful in our analysis. However, their growth in $N_0$ is higher than the $L^\infty_tH^s_x$ estimates.  

\begin{proof}[Proof of Proposition \ref{prop:energy}]
Now, we prove \eqref{eq:bound-w-conclusion} under the assumption \eqref{eq:bound-w-hypothesis}. Recall the equation for $\W$:
\EQ{
	i\pd_t \W + \frac12\De \W = t^{-\frac{1}{2}} |\U|\U.
}
Multiply the equation with $t^{\frac{1}{2}}\wb{\W_t}$, integrate in $x$, and take the real part, then we obtain 
\EQ{
	\frac12t^{\frac {1}2}\re\int_{\R^3} \De \W\wb\W_t \dd x = \re\int_{\R^3} |\U|\U \wb\W_t \dd x.
}
Next, using integration by parts,
\EQ{
	- \frac12t^{\frac {1}2}\re\int_{\R^3} \nabla \W\cdot\nabla\wb\W_t \dd x = \re\int_{\R^3} |\W|\W \wb\W_t \dd x + \re\int_{\R^3} (|\U|\U - |\W|\W) \wb\W_t \dd x,
}
which gives
\EQ{
- \frac14t^{\frac {1}2}\pd_t\brkb{\int_{\R^3} |\nabla \W|^2 \dd x} = \frac{1}{3} \pd_t\brkb{\int_{\R^3} |\W|^{3} \dd x} - \re\int_{\R^3} (|\U|\U - |\W|\W) \wb\W_t \dd x.
}
Thus,
\EQ{
\pd_t\E(t) = & \re\int_{\R^3} (|\U|\U - |\W|\W) \wb\W_t \dd x + \frac18t^{\frac {1}2} \int_{\R^3} |\nabla \W|^2 \dd x \\
\ge & \re\int_{\R^3} (|\U|\U - |\W|\W) \wb\W_t \dd x.
}
Now, we integrate from any $t\in I$ to $T_0$,
\EQ{
	\E(t) \loe \E(T_0) - \re\int_I\int_{\R^3}(|\U|\U - |\W|\W) \wb\W_\ta\dd x\dd \ta.
}
Taking supremum in $t\in I$,
\EQn{\label{eq:energy-1}
\sup_{t\in I}\E(t)\le \E(T_0) + |e|,
}
where the error $e=e(t_0,T_0,N_0)$ is defined by
\EQn{\label{defn:e}
e:= \re\int_{I}\int_{\R^3}(|\U|\U - |\W|\W) \wb\W_t\dd x\dd t.
}
Using $i\pd_t\W=-\De \W + |\U|\U$, we split the error term into three subcases:
\EQnnsub{
|e| \lsm &  \Big|\int_{I}\int_{\R^3} \chi_{\le0}\cdot(|\U|\U - |\W|\W) \De\wb{\W}\dd x\dd t\Big| \tag{Error I}\label{eq:energy-deltaw-w-v}\\
& + \Big|\int_{I}\int_{\R^3}\chi_{\ge0}\cdot(|\U|\U - |\W|\W) \De\wb{\W}\dd x\dd t\Big| \tag{Error II} \label{eq:energy-deltaw-others}\\
& + 
\Big|\int_I\int_{\R^3}|\U|^2|\V|(|\W|+|\V|)\dd x\dd t\Big|. \tag{Error III}\label{eq:energy-potential}
}

$\bullet$ \textbf{Estimate for \eqref{eq:energy-deltaw-w-v}.} We first make frequency decomposition
\EQnnsub{
\eqref{eq:energy-deltaw-w-v} \lsm & \sum_{N_1} \Big|\int_{I}\int_{\R^3} \De\wb{\W}_{N_1} \chi_{\le0}\cdot(|\U|\U-|\W+\V_{\le N_1}|(\W+\V_{\le N_1})) \dd x\dd t\Big| \tag{Error I: LH} \label{eq:energy-deltaw-w-v-LH}\\
& + \sum_{N_1} \Big|\int_{I}\int_{\R^3} \De\wb{\W}_{N_1} \chi_{\le0}\cdot(|\W+\V_{\le N_1}|(\W+\V_{\le N_1})-|\W|\W) \dd x\dd t\Big|. \tag{Error I: HL} \label{eq:energy-deltaw-w-v-HL}
}
For the first term \eqref{eq:energy-deltaw-w-v-LH}, we use the pointwise bound for the difference of nonlinear term $||\U|\U-|\W+\V_{\le N_1}|(\W+\V_{\le N_1})| \lsm |\V_{\ge N_1}| (|\W|+ |\V|+|\V_{\le N_1}|) $ and obtain
\EQnnsub{
\eqref{eq:energy-deltaw-w-v-LH} \lsm & \sum_{N_1} \int_{I}\int_{\R^3} |\De{\W}_{N_1} \chi_{\le0}\V_{\ge N_1}\W| \dd x\dd t  \label{eq:energy-deltaw-w-v-LH-w}\\
& + \sum_{N_1} \int_{I}\int_{\R^3} |\De{\W}_{N_1} \chi_{\le0}\V_{\ge N_1}|(|\V| + |\V_{\le N_1}|) \dd x\dd t. \label{eq:energy-deltaw-w-v-LH-v}
}
By Schur's test in Lemma \ref{lem:schurtest},  H\"older's inequality, and $\chi_{\le 0 } = \sum_{j\le 0} \chi_j = \sum_{j\le 0}\chi_j\wt \chi_j \wtt\chi_j$,
\EQn{\label{eq:energy-deltaw-w-v-LH-w-1} 
\eqref{eq:energy-deltaw-w-v-LH-w} \lsm & \sum_{N_1\le N_2} \int_{I}\int_{\R^3} |\De\W_{N_1} \chi_{\le0}\W\V_{N_2}|\dd x\dd t \\
\lsm &  \int_{I}\int_{\R^3} \sum_{N_1\le N_2} \frac{N_1^{1/2}}{N_2^{1/2}}|N_1^{-\frac12}\De\W_{N_1} \chi_{\le0} \W N_2^{\frac12}\V_{N_2}|\dd x\dd t \\
\lsm & \int_{I}\int_{\R^3} \normb{N_1^{-\frac12}\De\W_{N_1}}_{l_{N_1}^2} \chi_{\le0}|\W| \normb{N_2^{\frac12}\V_{N_2}}_{l_{N_2}^2}\dd x\dd t \\
\lsm & \sum_{j\le 0} 2^{\frac12j} \int_{I}\int_{\R^3} \normb{2^{-\frac12j}\chi_{j}N_1^{-\frac12}\De\W_{N_1}}_{l_{N_1}^2} 2^{\frac12j} \wt\chi_{j}|\W| \normb{2^{-\frac12j}\wtt\chi_{j}N_2^{\frac12}\V_{N_2}}_{l_{N_2}^2}\dd x\dd t \\
\lsm & \sup_{j\le 0}   \normb{2^{-\frac12j}\chi_{j}N_1^{-\frac12}\De\W_{N_1}}_{l_{N_1}^2L_{t,x}^2}  \normb{2^{\frac12j} \wt\chi_{j}\W}_{L_{t,x}^\I} \normb{2^{-\frac12j}\wtt\chi_{j}N_2^{\frac12}\V_{N_2}}_{l_{N_2}^2L_{t,x}^2}.
}
By \eqref{eq:energy-deltaw-w-v-LH-w-1}, Lemma \ref{lem:local-smoothing}, \eqref{eq:energy-w-x1}, \eqref{eq:initialdata-estimate-v}, and $N_2\gsm N_0$,
\EQn{\label{esti:energy-deltaw-w-v-LH-w}
\eqref{eq:energy-deltaw-w-v-LH-w} \lsm & \norm{\W}_{X^1} t_0^{-\frac12} N_0^{\frac12} \norm{P_{N_2}\V_{+}}_{l_{N_2}^2 L_x^2} \\
\lsm & t_0^{-3} T_0^2 N_0  \normb{N_2^{\frac12}P_{N_2}\V_{+}}_{l_{N_2}^2 L_x^2} \\
\lsm & t_0^{-3} T_0^2 \de_0 N_0. 
}
As for \eqref{eq:energy-deltaw-w-v-LH-v}, by H\"older's inequality, Lemma \ref{lem:schurtest}, \eqref{eq:nonl-esti-vnablav-frequency}, and \eqref{eq:energy-w-x1},
\EQn{\label{esti:energy-deltaw-w-v-LH-v}
& \eqref{eq:energy-deltaw-w-v-LH-v} \\
\lsm & \sum_{N_1\le N_2} \int_{I}\int_{\R^3} |\De{\W}_{N_1} \chi_{\le0}\V_{N_2}| (|\V| + |\V_{\le N_1}|) \dd x\dd t \\
\lsm &  \int_{I}\int_{\R^3} \normb{N_1^{-1}\De{\W}_{N_1}}_{l_{N_1}^2} \normb{N_2\V_{N_2}}_{l_{N_2}^2} (|\V| + |\V_{\le N_1}|) \dd x\dd t \\
\lsm & \norm{N_1^{-1}\De{\W}_{N_1}}_{l_{N_1}^2L_t^{4} L_x^{3} \cap l_{N_1}^2L_t^{\frac{4}{1-6\ep}}L_x^{\frac{3}{1+3\ep}}} \normb{\norm{N_2\V_{N_2}}_{l_{N_2}^2} (|\V| + |\V_{\le N_1}|)}_{L_t^{\frac43} L_x^{\frac32} + L_t^{\frac{4}{3+6\ep}}L_x^{\frac{3}{2-3\ep}}} \\
\lsm & \norm{\W}_{X^1}  t_0^{-\frac12} T_0^{\frac14+\frac32\ep} \de_0^2. \\
\lsm & t_0^{-3} T_0^3 \de_0^2 N_0.
}
Then, \eqref{esti:energy-deltaw-w-v-LH-w} and \eqref{esti:energy-deltaw-w-v-LH-v} give that
\EQn{\label{esti:energy-deltaw-w-v-LH}
\eqref{eq:energy-deltaw-w-v-LH} \lsm t_0^{-3} T_0^3 \de_0 N_0.
}

For the second term \eqref{eq:energy-deltaw-w-v-HL}, note that we have the pointwise bound
\EQ{
|\nabla(|\W+\V_{\le N_1}|(\W+\V_{\le N_1})-|\W|\W)| \lsm |\V_{\le N_1}\nabla\W| + |\nabla\V_{\le N_1}|(|\W| + |\V_{\le N_1}|),
}
then we apply the integration by parts and obtain
\EQnnsub{
\eqref{eq:energy-deltaw-w-v-HL} \lsm & \sum_{N_1\ge N_2} \int_{I}\int_{\R^3} |\nabla\W_{N_1}\cdot\nabla\chi_{\le0}\W \V_{N_2}| \dd x\dd t \label{eq:energy-deltaw-w-v-LH-nablachi}\\
& + \sum_{N_1\ge N_2} \int_{I}\int_{\R^3} |\chi_{\le0}\nabla\W_{N_1}\cdot\nabla\W \V_{N_2}| \dd x\dd t \label{eq:energy-deltaw-w-v-LH-nablaw}\\
& + \sum_{N_1\ge N_2} \int_{I}\int_{\R^3} |\chi_{\le0}\nabla\W_{N_1}\cdot\W \nabla\V_{N_2}| \dd x\dd t \label{eq:energy-deltaw-w-v-LH-nablav}\\
& + \sum_{N_1\ge N_2} \int_{I}\int_{\R^3} |\chi_{\le0}\nabla\W_{N_1}\cdot\V_{\le N_1} \nabla\V_{N_2}| \dd x\dd t. \label{eq:energy-deltaw-w-v-LH-nablav-v}
}
For \eqref{eq:energy-deltaw-w-v-LH-nablachi},  by H\"older's inequality and  the support set of $\nabla\chi_{\le 0}$, 
\EQn{\label{esti:energy-deltaw-w-v-LH-nablachi-1}
\eqref{eq:energy-deltaw-w-v-LH-nablachi} 
\lsm & \sum_{j:|j|\le2} \sum_{N_1\ge N_2} \int_{I}\int_{\R^3} |\chi_j\nabla\W_{N_1}\cdot\nabla\chi_{\le0}\W \V_{N_2}| \dd x\dd t \\
\lsm & T_0^{\frac14} \sum_{j:|j|\le2}\sum_{N_1\ge N_2}  \norm{\chi_j\nabla\W_{N_1}}_{L_{t,x}^2} \norm{\nabla\chi_{\le0}}_{L_x^\I} \norm{\W}_{L_t^\I L_x^6} \norm{\V_{N_2}}_{L_t^4 L_x^3}. 
}
By \eqref{eq:energy-w-h1}, Lemma \ref{lem:local-smoothing}, and \eqref{eq:energy-w-x1}, for any $j\in\Z$ with $|j|\le2$,
\EQn{\label{esti:energy-deltaw-w-v-LH-nablachi-2}
\norm{\chi_j\nabla\W_{N_1}}_{L_{t,x}^2} \lsm & T_0^{\frac12-\frac12\ep} \norm{\nabla\W_{N_1}}_{L_t^\I L_{x}^2}^{1-\ep} \normb{2^{-\frac12j}\chi_j\nabla\W_{N_1}}_{L_{t,x}^2}^{\ep} \\
\lsm & T_0^{\frac12-\frac12\ep} (t_0^{-\frac14} N_0^{\frac12})^{1-\ep}  \norm{\W_{N_1}}_{X^{\frac12}}^{\ep} \\
\lsm & N_1^{-\frac12\ep} T_0^{\frac12-\frac12\ep} (t_0^{-\frac14} N_0^{\frac12})^{1-\ep}  \norm{\W_{N_1}}_{X^{1}}^{\ep} \\
\lsm & N_1^{-\frac12\ep} T_0^{\frac12-\frac12\ep} (t_0^{-\frac14} N_0^{\frac12})^{1-\ep} (t_0^{-2} T_0^2 N_0)^\ep \\
\lsm & N_1^{-\frac12\ep} t_0^{-\frac12} T_0 N_0^{\frac12+\frac12\ep}.
}
Noting the fact that $\wh \V$ is supported on $\fbrk{\xi:|\xi|\gsm N_0}$, by \eqref{esti:energy-deltaw-w-v-LH-nablachi-1}, \eqref{esti:energy-deltaw-w-v-LH-nablachi-2}, \eqref{eq:energy-w-h1}, and \eqref{eq:linear-estimate-v},
\EQn{\label{esti:energy-deltaw-w-v-LH-nablachi}
\eqref{eq:energy-deltaw-w-v-LH-nablachi} \lsm & T_0^{\frac14} \sum_{N_1\ge N_2} N_1^{-\frac12\ep} t_0^{-\frac12} T_0 N_0^{\frac12+\frac12\ep} \cdot t_0^{-\frac14} N_0^{\frac12} \cdot N_0^{-\frac12} \normb{N_2^{\frac12}\V_{N_2}}_{L_t^4 L_x^3} \\
\lsm & t_0^{-1} T_0^{2} \de_0 N_0^{\frac12+\frac12\ep} \sum_{N_1\ge N_2} N_1^{-\frac12\ep} \\
\lsm & t_0^{-1} T_0^{2} \de_0 N_0^{\frac12+\frac12\ep} \sum_{N_2:N_2\gsm N_0} N_2^{-\frac12\ep} \\
\lsm & t_0^{-1} T_0^{2} \de_0 N_0^{\frac12}.
}
For \eqref{eq:energy-deltaw-w-v-LH-nablaw}, by H\"older's inequality,
\EQn{\label{eq:energy-deltaw-w-v-LH-nablaw-1}
\eqref{eq:energy-deltaw-w-v-LH-nablaw} 
\lsm & \sum_{j\le 0} \sum_{N_1\ge N_2} \int_{I}\int_{\R^3} |\chi_{j}\nabla\W_{N_1}\cdot\nabla\W \V_{N_2}| \dd x\dd t \\
\lsm & \sum_{j\le 0} \sum_{N_1\ge N_2}  \norm{\chi_j\nabla\W_{N_1}}_{L_t^2 L_x^{\frac{2}{1-\ep}}} \norm{\nabla\W}_{L_{t}^\I L_x^2} \norm{\V_{N_2}}_{L_t^2 L_x^{\frac2\ep}}.
}
By interpolation, \eqref{eq:energy-w-h1}, Lemma \ref{lem:local-smoothing}, and \eqref{eq:energy-w-x1}, for any $j\in\Z$ with $j\le0$,
\EQn{\label{eq:energy-deltaw-w-v-LH-nablaw-2}
\norm{\chi_j\nabla\W_{N_1}}_{L_t^2 L_x^{\frac{2}{1-\ep}}} \lsm & \norm{\chi_j\nabla\W_{N_1}}_{L_{t,x}^2}^{1-\frac32\ep} \norm{\nabla\W}_{L_t^2 L_x^6}^{\frac32\ep} \\
\lsm & 2^{\frac12\ep j} T_0^{\frac12-\frac54\ep} \norm{\nabla\W_{N_1}}_{L_t^\I L_{x}^2}^{1-\frac52\ep} \normb{2^{-\frac12j}\chi_j\nabla\W_{N_1}}_{L_{t,x}^2}^{\ep} \norm{\W}_{X^1}^{\frac32\ep} \\
\lsm & 2^{\frac12\ep j} T_0^{\frac12-\frac54\ep} (t_0^{-\frac14} N_0^{\frac12})^{1-\frac52\ep} N_1^{-\frac12\ep}\norm{\W_{N_1}}_{X^1}^{\ep} \norm{\W}_{X^1}^{\frac32\ep} \\
\lsm & 2^{\frac12\ep j} T_0^{\frac12-\frac54\ep} (t_0^{-\frac14} N_0^{\frac12})^{1-\frac52\ep} N_1^{-\frac12\ep} (t_0^{-2}T_0^2 N_0)^{ \frac52\ep} \\
\lsm & 2^{\frac12\ep j} N_1^{-\frac12\ep} t_0^{-\frac12} T_0 N_0^{\frac12+\frac54\ep}.
}
By \eqref{eq:linear-estimate-v} and the fact that $(2,\frac2\ep)$ is $\dot H_x^{1/2-3\ep/2}$-admissible,
\EQn{\label{eq:energy-deltaw-w-v-LH-nablaw-3}
\norm{\V_{N_2}}_{L_t^2 L_x^{\frac2\ep}} \lsm N_2^{-\frac32\ep} \normb{N_2^{\frac32\ep}\V_{N_2}}_{L_t^2 L_x^{\frac2\ep}} \lsm & N_0^{-\frac32\ep} \de_0.
}
By \eqref{eq:energy-deltaw-w-v-LH-nablaw-1},  \eqref{eq:energy-deltaw-w-v-LH-nablaw-2}, \eqref{eq:energy-deltaw-w-v-LH-nablaw-3}, and \eqref{eq:energy-w-h1},
\EQn{\label{esti:energy-deltaw-w-v-LH-nablaw}
\eqref{eq:energy-deltaw-w-v-LH-nablaw} \lsm & \sum_{j\le0} \sum_{N_1\ge N_2} 2^{\frac12\ep j} N_1^{-\frac12\ep} t_0^{-\frac12} T_0 N_0^{\frac12+\frac54\ep} \cdot t_0^{-\frac14} N_0^{\frac12} \cdot N_0^{-\frac32\ep} \de_0 \\
\lsm & t_0^{-1} T_0 \de_0 N_0^{1-\frac14\ep}  \sum_{N_1\ge N_2}  N_1^{-\frac12\ep}      \\
\lsm & t_0^{-1} T_0 \de_0 N_0^{1-\frac14\ep}  \sum_{N_2:N_2\gsm N_0}  N_2^{-\frac12\ep} \\
\lsm & t_0^{-1} T_0 \de_0 N_0^{1-\frac34\ep}.
}
The treatment for \eqref{eq:energy-deltaw-w-v-LH-nablav} is similar to \eqref{eq:energy-deltaw-w-v-LH}. By Schur's test in Lemma \ref{lem:schurtest}, H\"older's inequality, \eqref{eq:energy-w-linfty}, the local smoothing effect estimate in Lemma \ref{lem:local-smoothing}, and the fact that $N_2\gsm N_0$,
\EQn{\label{esti:energy-deltaw-w-v-LH-nablav}
\eqref{eq:energy-deltaw-w-v-LH-nablav} \lsm & \int_{I}\int_{\R^3} \sum_{N_1\ge N_2} \frac{N_2^{1/2}}{N_1^{1/2}}|N_1^{-\frac12}\nabla\W_{N_1} \chi_{\le0} \W N_2^{-\frac12}\nabla\V_{N_2}|\dd x\dd t \\
\lsm & \sum_{j\le 0} 2^{\frac12j} \int_{I}\int_{\R^3} \normb{2^{-\frac12j}\chi_{j}N_1^{\frac12}\nabla\W_{N_1}}_{l_{N_1}^2} 2^{\frac12j} \wt\chi_{j}|\W| \normb{2^{-\frac12j}\wtt\chi_{j}N_2^{-\frac12}\nabla\V_{N_2}}_{l_{N_2}^2}\dd x\dd t \\
\lsm & \sup_{j\le 0} \normb{2^{-\frac12j}\chi_{j}N_1^{\frac12}\nabla\W_{N_1}}_{l_{N_1}^2L_{t,x}^2}  \normb{2^{\frac12j} \wt\chi_{j}\W}_{L_{t,x}^\I} \normb{2^{-\frac12j}\wtt\chi_{j}N_2^{-\frac12}\nabla\V_{N_2}}_{l_{N_2}^2L_{t,x}^2} \\
\lsm & t_0^{-3} T_0^3 \de_0 N_0. 
}
As for  \eqref{eq:energy-deltaw-w-v-LH-nablav-v}, by H\"older's inequality, Lemma \ref{lem:local-smoothing}, \eqref{eq:linear-estimate-v}, and Lemma \ref{lem:linfty-littewoodpaley},
\EQn{\label{esti:energy-deltaw-w-v-LH-nablav-v}
\eqref{eq:energy-deltaw-w-v-LH-nablav-v} \lsm &  \int_{I}\int_{\R^3} \normb{\chi_{\le 0}N_1^{\frac12}\nabla\W_{N_1}}_{l_{N_1}^2}\sup_{N_1}|\V_{\le N_1}| \normb{N_2^{-\frac12}\nabla\V_{N_2}}_{l_{N_2}^2} \dd x\dd t \\
\lsm & \normb{\chi_{\le 0}N_1^{\frac12}\nabla\W_{N_1}}_{l_{N_1}^2L_{t,x}^2} \normb{\sup_{N_1}|\V_{\le N_1}|}_{L_t^\I L_x^3} \normb{N_2^{-\frac12}\nabla\V_{N_2}}_{l_{N_2}^2L_t^2 L_x^6} \\
\lsm & \norm{\W}_{X^{1}} \de_0^2 \\
\lsm & t_0^{-2} T_0^2 \de_0^2 N_0.
}
Combining \eqref{esti:energy-deltaw-w-v-LH-nablachi}, \eqref{esti:energy-deltaw-w-v-LH-nablaw}, \eqref{esti:energy-deltaw-w-v-LH-nablav}, and \eqref{esti:energy-deltaw-w-v-LH-nablav-v}, we have that
\EQn{\label{esti:energy-deltaw-w-v-HL}
\eqref{eq:energy-deltaw-w-v-HL} \lsm t_0^{-2} T_0^2 \de_0 N_0.
}
Now, \eqref{esti:energy-deltaw-w-v-LH} and \eqref{esti:energy-deltaw-w-v-HL} imply that
\EQn{\label{esti:energy-deltaw-w-v}
\eqref{eq:energy-deltaw-w-v} \lsm t_0^{-3} T_0^3 \de_0 N_0. 
}

$\bullet$ \textbf{Estimate for \eqref{eq:energy-deltaw-others}.} Integrating by parts, we have
\EQnnsub{
\eqref{eq:energy-deltaw-others} \lsm & \int_{I}\int_{\R^3}|\nabla\W \cdot \nabla\chi_{\ge0}  \V| (|\W| +|\V|) \dd x\dd t \tag{Error $\text{II}_1$}\label{eq:energy-deltaw-others-nablachi}\\
& + \int_{I}\int_{\R^3}\chi_{\ge0}|\nabla\W \cdot \nabla\W \V|\dd x\dd t \tag{Error $\text{II}_2$}\label{eq:energy-deltaw-others-nablaw}\\
& + \int_{I}\int_{\R^3}\chi_{\ge0}|\nabla\W \cdot \nabla\V \W|\dd x\dd t \tag{Error $\text{II}_3$}\label{eq:energy-deltaw-others-nablav-w}\\
& + \int_{I}\int_{\R^3}\chi_{\ge0}|\nabla\W \cdot \nabla\V \V|\dd x\dd t. \tag{Error $\text{II}_4$}\label{eq:energy-deltaw-others-nablav-v}
}
For \eqref{eq:energy-deltaw-others-nablachi}, by H\"older's inequality, \eqref{eq:energy-w-h1}, \eqref{eq:energy-w-l3}, and \eqref{eq:linear-estimate-v},
\EQn{\label{esti:energy-deltaw-others-nablachi}
\eqref{eq:energy-deltaw-others-nablachi} \lsm & T_0^{\frac12} \norm{\nabla\W}_{L_t^\I L_x^2} \norm{\nabla\chi_{\ge0}}_{L_x^\I} \norm{\V}_{L_t^2 L_x^6} (\norm{\W}_{L_t^\I L_x^3} + \norm{\V}_{L_t^\I L_x^3}) \\
\lsm & T_0^{\frac12} t_0^{-\frac14} N_0^{\frac12} \de_0 (N_0^{\frac13} + \de_0) \\
\lsm & t_0^{-\frac14} T_0^{\frac12} \de_0 N_0^{\frac56}.
}
For \eqref{eq:energy-deltaw-others-nablaw}, by H\"older's inequality, \eqref{eq:energy-w-h1}, and \eqref{eq:linear-estimate-v-radial},
\EQn{\label{esti:energy-deltaw-others-nablaw}
\eqref{eq:energy-deltaw-others-nablaw} \lsm T_0^{\frac12} \norm{\nabla\W}_{L_t^\I L_x^2}^2 \norm{\V}_{L_t^2 L_x^\I} \lsm \de_0 t_0^{-\frac12} T_0^{\frac12}N_0.
}
For \eqref{eq:energy-deltaw-others-nablav-w}, by the identity decomposition $\chi_{\ge1}=\sum_{j\ge1}\chi_j=\sum_{j\ge1}\wtt\chi_j\wt \chi_j \chi_j$ and H\"older's inequality,
\EQn{\label{eq:energy-deltaw-others-nablav-w-1}
\eqref{eq:energy-deltaw-others-nablav-w} \lsm & \sum_{j= 1}^{+\I} \int_I\int_{\R^3} |\chi_j\nabla\W\cdot \wt\chi_j\nabla\V||\wtt\chi_j\W|\dd x\dd t \\
\lsm & T_0^{\frac12} \sum_{j=J}^{+\I} 2^{-\ep j} \norm{\chi_j\nabla\W}_{L_t^\I L_x^2} \normb{2^{-\frac12j} \wt\chi_j\nabla\V}_{L_{t,x}^2} \normb{2^{(\frac12+\ep)j}\wtt\chi_j\W}_{L_{t,x}^\I}.
}
By the radial Sobolev's inequality in Lemma \ref{lem:radial-sobolev},
\EQ{
	\normb{2^{(\frac12+\ep)j}\wtt\chi_j\W}_{L_{x}^\I} \lsm \normb{|\nabla|^{1-2\ep}\W}_{L_{x}^{\frac{6}{3-2\ep}}}.
}
By Gagliardo-Nirenberg's inequality in Lemma \ref{lem:GN},
\EQ{
	\normb{|\nabla|^{1-2\ep}\W}_{L_{x}^{\frac{6}{3-2\ep}}} \lsm \norm{\W}_{L_x^3}^{2\ep} \norm{\W}_{\dot H_x^1}^{1-2\ep}.
}
Combining \eqref{eq:energy-deltaw-others-nablav-w-1} and the above two inequalities,
\EQn{\label{eq:energy-deltaw-others-nablav-w-2}
\eqref{eq:energy-deltaw-others-nablav-w} \lsm &  T_0^{\frac12}  \de_0 \norm{\nabla\W}_{L_t^\I L_x^2} \norm{\W}_{L_t^\I L_x^3}^{2\ep} \norm{\W}_{L_t^\I \dot H_x^{1}}^{1-2\ep}.
}
Then, by \eqref{eq:energy-deltaw-others-nablav-w-2}, \eqref{eq:energy-w-h1}, and \eqref{eq:energy-w-x1/2},
\EQn{\label{esti:energy-deltaw-others-nablav-w}
\eqref{eq:energy-deltaw-others-nablav-w} \lsm  T_0^{\frac12} \de_0 (N_0^{\frac13})^{2\ep} (t_0^{-\frac14}N_0^{\frac12})^{2-2\ep}  \lsm  t_0^{-1} T_0^{\frac12} \de_0 N_0^{1-\frac13\ep}.
}
For \eqref{eq:energy-deltaw-others-nablav-v}, by H\"older's inequality, \eqref{eq:nonl-esti-vnablav}, and \eqref{eq:energy-w-x1},
\EQn{\label{esti:energy-deltaw-others-nablav-v}
\eqref{eq:energy-deltaw-others-nablav-v} \lsm &  \int_{I}\int_{\R^3} |\nabla{\W} \nabla\V \V| \dd x\dd t \\
\lsm & \norm{\nabla \W}_{L_t^{4} L_x^{3} \cap L_t^{\frac{4}{1-6\ep}}L_x^{\frac{3}{1+3\ep}}} \normb{\nabla\V \V}_{L_t^{\frac43} L_x^{\frac32} + L_t^{\frac{4}{3+6\ep}}L_x^{\frac{3}{2-3\ep}}} \\
\lsm & \norm{\W}_{X^1}  t_0^{-\frac12} T_0^{\frac14+\frac32\ep} \de_0^2. \\
\lsm & t_0^{-3} T_0^3 \de_0^2 N_0.
}
Combining  \eqref{esti:energy-deltaw-others-nablachi}, \eqref{esti:energy-deltaw-others-nablaw}, \eqref{esti:energy-deltaw-others-nablav-w}, and \eqref{esti:energy-deltaw-others-nablav-v},
\EQn{\label{esti:energy-deltaw-others}
\eqref{eq:energy-deltaw-others} \lsm & t_0^{-3} T_0^3 \de_0 N_0.
}

$\bullet$ \textbf{Estimate for \eqref{eq:energy-potential}.} 
By H\"older's inequality, \eqref{eq:energy-w-l3}, \eqref{eq:linear-estimate-v}, and \eqref{eq:linear-estimate-v-radial},
\EQn{\label{esti:energy-potential}
	\eqref{eq:energy-potential} \lsm & \int_{t_0}^{T_0}\int_{\R^3}|\V|(|\W|^3+|\V|^3)\dd x\dd t \\
	\lsm & \norm{\V}_{L_t^1 L_x^\I} \norm{\W}_{L_t^\I L_x^3}^3 +  \norm{\V}_{L_{t,x}^4}^4 \\
	\lsm & T_0^{\frac12} \norm{\V}_{L_t^2 L_x^\I} N_0 + T_0^{\frac12} \norm{\V}_{L_t^8L_x^4}^4 \\
	\lsm & \de_0 T_0^{\frac12} N_0 +\de_0^4 T_0^{\frac12}.
}

Finally, combining \eqref{eq:energy-assumption}, \eqref{eq:energy-1}, \eqref{esti:energy-deltaw-w-v}, \eqref{esti:energy-deltaw-others}, and \eqref{esti:energy-potential},
\EQ{
\sup_{t\in I}\E(t)\le AN_0 + C(A) \de_0  t_0^{-3} T_0^{3} N_0.
}
Then, \eqref{eq:bound-w-conclusion} follows by taking $\de_0=\de_0(A,T_0,t_0)$ as
\EQn{\label{defn:delta0}
\de_0 := \min\fbrkbb{\eta,\brkb{ \frac{t_0^{\frac12}}{8C}}^{\frac12},\brkb{\frac{At_0^{3}}{2C(A)T_0^{3}}}^{\frac13}},
}
where $\eta>0$ denotes some sufficiently small global constant, and the constant $\brkb{ \frac{t_0^{1/2}}{8C}}^{1/2}$ appears in \eqref{eq:delta0} above. Note that $A=A(\norm{\U_+}_{\dot H_x^{1/2}(\R^3)})$, $T_0=T_0(\de_1,\U_+)$, and $\de_1>0$ denotes some sufficiently small absolute constant appeared in Proposition \ref{prop:localI}, then the above choice of $\de_0$ is well-defined, and $\de_0$ depends on $t_0$ and $u_0$.
\end{proof}
\subsection{Global well-posedness}\label{sec:gwp}
Now, we give the proof of global well-posedness by standard extension argument. 
\begin{prop}[Global well-posedness]\label{prop:gwp}
Let Reduction \ref{assu:main} hold, $\de_1$ and $T_0$ be given in Proposition \ref{prop:localI}, $A$ be given in Corollary \ref{cor:energy-estimate-t0}. Then, the equation \eqref{eq:nls-pc} admits a unique solution 
\EQ{
\U(t)\in S(t)\V_+ + C([t_0,\I); \dot H_x^{\frac12} \cap \dot H_x^1(\R^3)),
}
satisfying
\EQ{
\lim_{t\ra+\I} \norm{\U-S(t)\U_+}_{\dot H_x^{1/2}(\R^3)} = 0.
}
Moreover, we have
\EQn{\label{eq:energy}
\sup_{t\in[t_0,T_0]} \E(t) \le C(t_0,u_0).
}
\end{prop}
\begin{proof}
In the view of Propositions \ref{prop:localI} and \ref{prop:localII}, it suffices to extending the solution backwards from $T_0$ to $t_0$. By Corollary \ref{cor:energy-estimate-t0},
\EQ{
\E(T_0) \le A N_0 \le 2A N_0,
}
then
\EQ{
\norm{\W(T_0)}_{\dot H_x^1(\R^3)} \le \brkb{8A T_0^{-\frac12} N_0}^{\frac12} \le \brkb{8A t_0^{-\frac12} N_0}^{\frac12}=:E_0.
}
By \eqref{eq:local-h1/2}, we also have $\norm{\W(t_0)}_{\dot H^{1/2}} \le C$. Recall the equation for $\W$ in \eqref{eq:nls-w-localIII}:
\EQ{
\W(t)= & S(t-t_1)\W(t_1)-i\int_{t_1}^t  \ta^{-\frac12}|\U(\ta)|\U(\ta)\dd \ta.
}
Applying Proposition \ref{prop:localIII} for \eqref{eq:nls-w-localIII} with $t_1=T_0$, there exists $t_2'=t_2'(E_0,t_0)$ such that \eqref{eq:nls-w-localIII} admits a unique solution $\W\in C([T_0-t_2',T_0];\dot H_x^{1/2}\cap\dot H_x^1(\R^3))$. Next, applying Proposition \ref{prop:energy}, we have that
\EQ{
\sup_{t\in [T_0-t_2',T_0]}\E(t) \le 2A N_0.
}
Particularly, we also have that
\EQ{
\norm{\W(T_0-t_2')}_{\dot H_x^1(\R^3)} \le \brkb{8A (T_0-t_2')^{-\frac12} N_0}^{\frac12} \le \brkb{8A t_0^{-\frac12} N_0}^{\frac12} = E_0.
}
Furthermore, applying 
\eqref{eq:energy-w-x1/2} on $I=[T_0-t_2',T_0]$ with starting time $T_0$ and combining the above energy bound,
\EQn{\label{eq:gwp-x1/2-firstinterval}
\norm{\W}_{X^{\frac12}} &\lsm  \norm{\W(T_0)}_{\dot H_x^{\frac12}} + \normb{t^{-\frac12}|\nabla|^{\frac12}(|\U|\U)}_{L_t^{\frac43} L_x^{\frac32}} \\
&\lsm  1 + t_0^{-\frac12} |I|^{\frac12} \brkb{\de_0 + \norm{\W}_{L_t^\I L_x^3} }  \de_0 \\
& + t_0^{-\frac12}|I|^{\frac34} \brkb{\de_0 +\norm{\W}_{L_t^\I L_x^3}} \normb{|\nabla|^{\frac12}\W}_{L_t^\I L_x^3} \\
&\lsm_A  1 + t_0^{-\frac12} T_0^{\frac12} (\de_0 + N_0^{\frac13}) \de_0 + t_0^{-1} T_0^{\frac34} (\de_0 + N_0^{\frac13}) N_0^{\frac12} \\
&\lsm_A  t_0^{-1} T_0^{\frac34} N_0^{\frac56} =:E_1,
}
where we use the upper bound $|I|\le T_0$. Particularly, $\norm{\W(T_0-t_2')}_{\dot H_x^{1/2}} \le E_1$. Then, we can continue applying Proposition \ref{prop:localIII} starting from $T_0-t_2'$ and extending the solution on $[T_0-t_2'-t_2'',T_0-t_2']$ with some $t_2''=t_2''(t_0,E_0,E_1)$. (We overlook the technical treatment in the case when $T_0-t_2'-t_2''<t_0$.)

On this interval, we can see the energy bound is still that $\norm{\W(T_0-t_2')}_{\dot H_x^1(\R^3)} \le E_0$ due to the a priori bound in Proposition \ref{prop:energy}. Moreover, the $\dot H_x^{1/2}$-estimate is also global: in  \eqref{eq:gwp-x1/2-firstinterval}, we modify the interval as $I=[T_0-t_2'-t_2'',T_0]$ with starting time $T_0$, then the implicit constant in the estimate does not rely on $t_2'$ and $t_2''$. Therefore, we can still derive that
\EQ{
\norm{\W(T_0-t_2'-t_2'')}_{\dot H_x^{1/2}} \le E_1.
} 

Now, we are able to extending the solution to $t_0$ by induction, and the length of interval in every step is $t_2''$, which is independent of the starting time of every step. Then, we can obtain $\W\in C([t_0,T_0];\dot H_x^{1/2}\cap\dot H_x^1(\R^3))$.

Finally, since $A=A(\norm{\U_+}_{\dot H_x^{1/2}(\R^3)})$, $T_0=T_0(\de_1,\U_+)$, and $\de_1>0$ denotes some sufficiently small absolute constant in Proposition \ref{prop:localI}, $\de_0=\de_0(t_0,u_0)$ by \eqref{defn:delta0}, and $N_0=N_0(\de_0)$, then the upper bound of $\E(t)$, namely $2AN_0$, depends on $t_0$ and $u_0$.
\end{proof}

\subsection{Proof of Theorem \ref{thm:3dquadratic}}\label{sec:maintheorem}
It suffices to consider the forward-in-time case, namely for any $T>0$, $e^{-\frac12it\De}u \in u_0 + C([0,T];\F \dot H_x^{1/2} \cap \F \dot H_x^{1}(\R^3))$. Then, this follows by applying the inverse pseudo conformal energy transform to the result in Proposition \ref{prop:gwp} under the Reduction \ref{assu:main} with $t_0=1/T$.

\appendix

\vspace{2cm}

\section{Proof of the local well-posedness away from the origin}

\vspace{0.5cm}

Now, we give the proof of Proposition \ref{thm:lwp}. We only consider the $t_0>0$  case. Here, the spacetime variable $(t,x)$ is restricted on $[t_0-T,t_0+T]\times\R^3$, for some $0<T<t_0$ that will be defined later. Let $I:=[t_0-T,t_0+T]$. Recall that $S(-t_0)u(t_0) \in \F\dot H_x^s(\R^3)$, and the integral equation reads
\EQ{
u(t) = S(t-t_0)u(t_0) - i \int_{t_0}^t S(t-\ta)(|u(\ta)|u(\ta)) \dd \ta.
}
Apply the vector field, then
\EQ{
|J(t)|^su(t) = |J(t)|^sS(t-t_0)u(t_0) - i \int_{t_0}^t S(t-\ta)|J(\ta)|^s(|u(\ta)|u(\ta)) \dd \ta.
}
First, by the definition of the fractional vector field and the Strichartz estimate \eqref{eq:strichartz-1}, for any $L^2$-admissible pair $(q,r)$,
\EQn{\label{app:linear}
\norm{|J(t)|^sS(t-t_0)u(t_0)}_{L_t^q L_x^r} \lsm \norm{S(t)|x|^sS(-t_0)u(t_0)}_{L_t^q L_x^r} \lsm \norm{S(-t_0)u(t_0)}_{\F\dot H_x^s}.
}
For the integral term, let $v=M(-t)u$, where $M(t):=e^{\frac{i|x|^2}{2t}}$, then by \eqref{eq:strichartz-2}, H\"older's inequality, and $|J(t)|^s=M(t)|t\nabla|M(-t)$,
\EQn{\label{app:nonlinear-1}
\normb{\int_{t_0}^t S(t-\ta)|J(\ta)|^s(|u(\ta)|u(\ta)) \dd \ta}_{L_t^q L_x^r} \lsm & \norm{|J(t)|^s(|u|u)}_{L_t^1 L_x^2} \\
\sim & \norm{|t\nabla|^s(|v|v)}_{L_t^1 L_x^2} \\
\lsm & \norm{v}_{L_t^{\frac43}L_x^6} \norm{|t\nabla|^sv}_{L_t^4 L_x^3} \\
\lsm & \norm{v}_{L_t^{\frac43}L_x^6} \norm{|J(t)|^su}_{L_t^4 L_x^3}.
}
Now, we temporarily take $T$ such that
\EQ{
0<T\le \min\fbrk{\frac{t_0}{2},1},
}
then $t\sim t_0$ on $t\in I$. Therefore, by Sobolev's inequality in $x$ and H\"older's inequality in $t$,
\EQn{\label{app:nonlinear-2}
\norm{v}_{L_t^{\frac43}L_x^6} \lsm t_0^{-s} \norm{|t\nabla|^sv}_{L_t^{\frac43}L_x^{\frac{6}{1+2s}}} 
\lsm t_0^{-s} T^{\frac{1+2s}{4}} \norm{|J(t)|^su}_{L_t^{\frac{2}{1-s}}L_x^{\frac{6}{1+2s}}}.
}
By \eqref{app:nonlinear-1} and \eqref{app:nonlinear-2}, we have that
\EQn{\label{app:nonlinear}
&\normb{\int_{t_0}^t S(t-\ta)|J(\ta)|^s(|u(\ta)|u(\ta)) \dd \ta}_{L_t^q L_x^r} \\
\lsm & t_0^{-s} T^{\frac{1+2s}{4}} \norm{|J(t)|^su}_{L_t^{\frac{2}{1-s}}L_x^{\frac{6}{1+2s}}} \norm{|J(t)|^su}_{L_t^4 L_x^3}.
}
Note that $(\frac{2}{1-s},\frac{6}{1+2s})$ and $(4,3)$ are $L^2$-admissible, and $t_0$ is a fixed constant. Now, in the view of \eqref{app:linear} and \eqref{app:nonlinear}, take $T$ sufficiently small depending on $\norm{S(-t_0)u(t_0)}_{\F\dot H_x^s}$, then the local well-posedness holds by standard contraction-mapping argument in the resolution space
\EQ{
\fbrk{|J(t)|^su\in C(I;L_x^2(\R^3)) : \norm{|J(t)|^su}_{L_t^\I L_x^2\cap L_t^{\frac{2}{1-s}}L_x^{\frac{6}{1+2s}}\cap L_t^4 L_x^3(I\times\R^3)} \le 2R},
}
where $R:=C\norm{S(-t_0)u(t_0)}_{\F\dot H_x^s(\R^3)}$. This completes the proof of Proposition \ref{thm:lwp}.

\end{document}